%% file: stars_arXiv.tex
\documentclass{article}

\usepackage{color,hyperref,amsthm}
\usepackage{doi}
\usepackage[numbers,compress,sort]{natbib} 
\hypersetup{
unicode = false,
pdftoolbar = true,
pdfmenubar = true,
pdffitwindow = true,
pdftitle = {Scalable-PB},
pdfauthor = {Dzahini},
pdfsubject = {Scalable},
pdfnewwindow = true,
pdfkeywords = {Scalable, Subspace},
colorlinks = true,
linkcolor = blue,
citecolor = blue,
filecolor = black,
urlcolor = blue,
breaklinks = true
}

\usepackage{algorithmic}
\usepackage[noresetcount,lined,boxed]{algorithm2e} 
\newtheorem{definition}{Definition}[section]
\newtheorem{theorem}{Theorem}[section]
\newtheorem{lemma}{Lemma}[section]

\newtheorem{corollary}{Corollary}[section]

\newtheorem{remark}{Remark}[section]
\usepackage[top=1in,bottom=1in,left=1in,right=1in]{geometry} 

\input{ex_shared}

\title{Stochastic trust-region algorithm in random subspaces with convergence and expected complexity analyses\thanks{This work was supported in part by the U.S.~Department of Energy, Office of
Science, Office of Advanced Scientific Computing Research Applied Mathematics
Program under Contract No.~DE-AC02-06CH11357.}}

\author{
	\href{mailto:kdzahini@anl.gov}{Kwassi Joseph Dzahini}\thanks{Argonne National Laboratory, 9700 S. Cass Avenue, Lemont, IL 60439, USA (\href{https://www.anl.gov/profile/kwassi-joseph-dzahini}{ www.anl.gov/profile/kwassi-joseph-dzahini},
	 \href{http://www.mcs.anl.gov/\%7Ewild/}{ www.mcs.anl.gov/$_{\widetilde{~}}$wild/}).
	}
	\and
	\href{mailto:wild@anl.gov}{Stefan M. Wild}\footnotemark[2]
}

\begin{document}

\maketitle

\begin{abstract}
This work proposes a framework for large-scale stochastic derivative-free optimization (DFO) by introducing STARS, a trust-region method based on iterative minimization in random subspaces. This framework is both an algorithmic and theoretical extension of an algorithm for stochastic optimization with random models (STORM). 
Moreover, STARS achieves scalability by minimizing interpolation models that approximate the objective in low-dimensional
affine subspaces, thus significantly reducing per-iteration costs in terms of function evaluations and yielding strong performance on large-scale stochastic DFO problems. The user-determined dimension of these subspaces, when the latter are defined, for example, by the columns of so-called Johnson--Lindenstrauss transforms, turns out to be independent of the dimension of the problem. For convergence purposes, both a particular quality of the subspace and the accuracies of random function estimates and models are required to hold with sufficiently high, but fixed, probabilities. Using martingale theory under the latter assumptions, an almost sure global convergence of STARS to a first-order stationary point is shown, and the expected number of iterations required to reach a desired first-order accuracy is proved to be similar to that of STORM and other stochastic DFO algorithms, up to constants. 
\end{abstract}

\input{body}

\end{document}

%% file: ex_shared.tex
\usepackage{amsfonts}
\usepackage{graphicx}
\usepackage{graphics}
\usepackage{adjustbox}
\usepackage{epstopdf}
\usepackage{algorithmic}
\usepackage{booktabs}
\ifpdf
  \DeclareGraphicsExtensions{.eps,.pdf,.png,.jpg}
\else
  \DeclareGraphicsExtensions{.eps}
\fi

\usepackage{amsopn}

\usepackage{amssymb,amsmath}
\usepackage{dsfont}
\usepackage{mathrsfs} 

\usepackage{listings}
\definecolor{codegreen}{rgb}{0,0.6,0}
\definecolor{codegray}{rgb}{0.5,0.5,0.5}
\definecolor{codepurple}{rgb}{0.58,0,0.82}
\definecolor{backcolour}{rgb}{0.95,0.95,0.92}
\definecolor{ForestGreen}{RGB}{34,139,34}

\usepackage{xcolor}
\usepackage{listings}

\usepackage{xparse}

\usepackage{stackengine}  
\newcommand{\ubar}[1]{\stackunder[1.0pt]{$#1$}{\rule{.8ex}{.075ex}}} 

\NewDocumentCommand{\codeword}{v}{%
	\texttt{\textcolor{blue}{#1}}%
}

\lstdefinestyle{mystyle}{
	backgroundcolor=\color{backcolour},   
	commentstyle=\color{codegreen},
	keywordstyle=\color{magenta},
	numberstyle=\tiny\color{codegray},
	stringstyle=\color{codepurple},
	basicstyle=\ttfamily\footnotesize,
	breakatwhitespace=false,         
	breaklines=true,                 
	captionpos=b,                    
	keepspaces=true,                 
	numbers=left,                    
	numbersep=5pt,                  
	showspaces=false,                
	showstringspaces=false,
	showtabs=false,                  
	tabsize=2
}

\usepackage{titlesec}
\usepackage{mdframed}  
\usepackage{colortbl}
\definecolor{highlight}{HTML}{DFEFF9}
\mdfdefinestyle{highlightbox}{leftmargin=0pt,rightmargin=0pt,%
	innerleftmargin=3pt,innerrightmargin=2pt,%
	innertopmargin=4pt,innerbottommargin=4pt,%
	backgroundcolor=highlight,linecolor=white}
\definecolor{algm}{HTML}{D4EFDF} 
\mdfdefinestyle{highlightalg}{leftmargin=0pt,rightmargin=0pt,%
	innerleftmargin=3pt,innerrightmargin=2pt,%
	innertopmargin=4pt,innerbottommargin=4pt,%
	backgroundcolor=algm,linecolor=white}
\definecolor{soft}{HTML}{F2D7D5} 
\mdfdefinestyle{highlightstruct}{leftmargin=0pt,rightmargin=0pt,%
	innerleftmargin=3pt,innerrightmargin=2pt,%
	innertopmargin=4pt,innerbottommargin=4pt,%
	backgroundcolor=soft,linecolor=white}

\usepackage{soul}

\makeatletter
\def\namedlabel#1#2{\begingroup
	#2%
	\def\@currentlabel{#2}%
	\phantomsection\label{#1}\endgroup
}
\makeatother

\newcommand{\Reals}{\mathbb{R}} 
\newcommand{\R}{\Reals} 
\newcommand{\Sym}{\mathbb{S}} 

\DeclareMathOperator*{\argmin}{arg\,min}

\newcommand{\B}[2]{\mathcal{B}(#1 ; #2 )} 
 
\newcommand{\E}[2][]{\operatorname{\mathbb{E}}_{#1}\left[ #2\right]} 
 
\usepackage{bm} 

\newcommand{\Ab}{\bm{A}}

\newcommand{\xb}{\bm{x}}

\newtheorem{assumption}{Assumption}

\usepackage{tikz}
\usetikzlibrary{positioning}
\usetikzlibrary{calc}
\usetikzlibrary{shapes.geometric}
\usetikzlibrary{topaths}
\usetikzlibrary{external}
\usetikzlibrary{fadings,decorations.pathreplacing} 

\usepackage{pgfplots}
\usetikzlibrary{shapes,decorations.pathmorphing,patterns}
\pgfplotsset{compat=newest}
\usepgfplotslibrary{fillbetween}

\graphicspath{{figures/}}

\newcommand{\Z}{\mathbb{Z}}
\newcommand{\Esp}{\mathbb{E}}

\newcommand{\N}{\mathbb{N}}
\newcommand{\V}{\mathbb{V}}

\newcommand{\F}{\mathcal{F}}

\newcommand{\n}{\mathcal{N}}

\newcommand{\pr}{\mathbb{P}}

\newcommand{\ds}{\displaystyle}

\newcommand{\dmaxx}{\delta_{\max}}

\newcommand{\ijk}{\mathds{1}_{J_k^Q}}

\newcommand{\ijkc}{\mathds{1}_{\bar{J}_k^Q}}
 
\newcommand{\iik}{\mathds{1}_{I_k^Q}}

\newcommand{\iikc}{\mathds{1}_{\bar{I}_k^Q}}

\newcommand{\abs}[1]{\left\lvert#1\right\rvert}
\newcommand{\intbracket}[1]{\left[\!\left[#1\right]\!\right]}
\newcommand{\accolade}[1]{\left\lbrace#1\right\rbrace}
\newcommand{\accoladekinN}[1]{{\left\lbrace#1\right\rbrace}_{k\in\mathbb{N}}}

\newcommand{\prob}[1]{\mathbb{P}\left(#1\right)}
\newcommand{\var}[1]{\mathbb{V}\left(#1\right)}

\newcommand{\norme}[1]{\left\lVert#1\right\rVert}

\newcommand{\normii}[1]{{\left\lVert#1\right\rVert}_{2}}

\newcommand{\comb}[2]{\begin{bmatrix}
{#1}\\
{#2}
\end{bmatrix}}

\newcommand{\Sk}{\bm{{\stackunder[1.0pt]{$s$}{\rule{.8ex}{.075ex}}}_k}}
\newcommand{\sk}{\bm{s_k}}
\newcommand{\rn}{\mathbb{R}^n}
\newcommand{\rp}{\mathbb{R}^p}
\newcommand{\rnn}{\mathbb{R}^{n\times n}}
\newcommand{\rnp}{\mathbb{R}^{n\times p}}
\newcommand{\rpn}{\mathbb{R}^{p\times n}}
\newcommand{\rpp}{\mathbb{R}^{p\times p}}
\newcommand{\sRandom}{\bm{{\stackunder[1.0pt]{$s$}{\rule{.8ex}{.075ex}}}}}
\newcommand{\s}{\bm{s}}
\newcommand{\viRandom}{\bm{{\stackunder[1.0pt]{$v$}{\rule{0.8ex}{.075ex}}}}}
\newcommand{\vi}{\bm{v}}
\newcommand{\SMatrix}{\bm{{\stackunder[1.0pt]{$S$}{\rule{1.3ex}{.075ex}}}}}
\newcommand{\RMatrix}{\bm{{\stackunder[1.0pt]{$R$}{\rule{1.3ex}{.075ex}}}}}
\newcommand{\RrandomVar}{{\stackunder[1.0pt]{$r$}{\rule{0.8ex}{.075ex}}}}
\newcommand{\SrandomVar}{{\stackunder[1.0pt]{$s$}{\rule{0.8ex}{.075ex}}}}
\newcommand{\etarandomVar}{{\stackunder[1.0pt]{$\eta$}{\rule{0.8ex}{.075ex}}}}
\newcommand{\sigmarandomVar}{{\stackunder[1.0pt]{$\sigma$}{\rule{0.8ex}{.075ex}}}}
\newcommand{\Qk}{\bm{Q_k}} 
\newcommand{\QkRandom}{\bm{{\stackunder[0.6pt]{$Q$}{\rule{1.3ex}{.075ex}}}_k}}
\newcommand{\QzeroRandom}{\bm{{\stackunder[0.6pt]{$Q$}{\rule{1.3ex}{.075ex}}}_0}}
\newcommand{\QOneRandom}{\bm{{\stackunder[0.6pt]{$Q$}{\rule{1.3ex}{.075ex}}}_1}}
\newcommand{\QkmunRandom}{\bm{{\stackunder[0.6pt]{$Q$}{\rule{1.3ex}{.075ex}}}_{k-1}}}
\newcommand{\QellRandom}{\bm{{\stackunder[0.6pt]{$Q$}{\rule{1.3ex}{.075ex}}}_{\ell}}}
\newcommand{\QkTranspose}{\bm{Q_k}^{\top}} 
\newcommand{\QkRandomTranspose}{\bm{{\stackunder[0.6pt]{$Q$}{\rule{1.3ex}{.075ex}}}_k}^{\top}}
 
\newcommand{\Mhatk}{\hat{{{\stackunder[1.1pt]{$m$}{\rule{1.1ex}{.09ex}}}}}_k} 
\newcommand{\mhatk}{\hat{m}_k} 

\newcommand{\Mhatell}{\hat{{{\stackunder[1.1pt]{$m$}{\rule{1.1ex}{.09ex}}}}}_{\ell}}

\newcommand{\fok}{f_k^0}
\newcommand{\Fok}{{\stackunder[0.6pt]{$f$}{\rule{.8ex}{.075ex}}}_k^0}
\newcommand{\Foell}{{\stackunder[0.6pt]{$f$}{\rule{.8ex}{.075ex}}}_{\ell}^0}

\newcommand{\fsk}{f_k^s}
\newcommand{\Fsk}{{\stackunder[0.6pt]{$f$}{\rule{.8ex}{.075ex}}}_k^s}
\newcommand{\Fsell}{{\stackunder[0.6pt]{$f$}{\rule{.8ex}{.075ex}}}_{\ell}^s}

\newcommand{\dkun}{{\delta_{k+1}}}

\newcommand{\iak}{\mathds{1}_{\mathcal{A}_k}}
\newcommand{\Q}{\bm{Q}} 
\newcommand{\Qmax}{Q_{\max}} 
\newcommand{\hopt}{h_{\rm opt}} 
\newcommand{\szero}{\bm{s_0}} 
\newcommand{\sone}{\bm{s_1}} 
\newcommand{\stwo}{\bm{s_2}} 
\newcommand{\spi}{\bm{s_p}}

\newcommand{\sii}{\bm{s_i}}
\newcommand{\SSk}{\mathbb{S}^k}

\newcommand{\Lkhat}{\bm{\hat{L}_k}}
\newcommand{\Xk}{\bm{{\stackunder[1.0pt]{$x$}{\rule{1.0ex}{.075ex}}}_k}}
\newcommand{\X}{{\stackunder[1.0pt]{$x$}{\rule{1.0ex}{.075ex}}}}
\newcommand{\xk}{\bm{{x}_k}}
\newcommand{\xkun}{\bm{{x}_{k+1}}}

\newcommand{\x}{\bm{x}}
\newcommand{\di}{\bm{d}}
\newcommand{\dik}{\bm{d_k}}

\newcommand{\kegPrimeSquare}{\kappa_{eg}^{\prime\, 2}}

\newcommand{\Lg}{L_g}

\newcommand{\Gkhat}{\bm{\hat{{{\stackunder[1.0pt]{$g$}{\rule{1.0ex}{.075ex}}}}}_k}}
\newcommand{\Gk}{{{\stackunder[1.0pt]{$g$}{\rule{1.0ex}{.075ex}}}}_k}
\newcommand{\gkhat}{\bm{\hat{g}_k}}
\newcommand{\fmqalgebra}{\mathcal{F}^{\hat{m}\cdot Q}_{k-1}}
\newcommand{\falgebra}{\mathcal{F}_{k-1}}
\newcommand{\fqalgebra}{\mathcal{F}_{k-1}^Q}
\newcommand{\ef}{\varepsilon_f}
\newcommand{\Dk}{{\stackunder[1.0pt]{$\delta$}{\rule{.8ex}{.075ex}}}_k}
\newcommand{\Dkun}{{\stackunder[1.0pt]{$\delta$}{\rule{.8ex}{.075ex}}}_{k+1}}
\newcommand{\Dkmun}{{\stackunder[1.0pt]{$\delta$}{\rule{.8ex}{.075ex}}}_{k-1}}
\newcommand{\dk}{\delta_k}
\newcommand{\Hkhat}{\bm{\hat{H}}_k}
\newcommand{\HkhatRandom}{\hat{{\stackunder[1.0pt]{$H$}{\rule{1.3ex}{.075ex}}}}_k}
\newcommand{\igkhat}{\mathds{1}_{G_k}}
\newcommand{\igkhatc}{\mathds{1}_{\bar{G}_k}}
\newcommand{\igk}{\mathds{1}_{\Gamma_k}}
\newcommand{\igkc}{\mathds{1}_{\bar{\Gamma}_k}}
\newcommand{\yOne}{{\stackunder[0.6pt]{$y$}{\rule{1.0ex}{.075ex}}}_1}
\newcommand{\yTwo}{{\stackunder[0.6pt]{$y$}{\rule{1.0ex}{.075ex}}}_2}
\newcommand{\ypi}{{\stackunder[0.6pt]{$y$}{\rule{1.0ex}{.075ex}}}_p}
\newcommand{\yii}{{\stackunder[0.6pt]{$y$}{\rule{1.0ex}{.075ex}}}_i}
\newcommand{\azeroRandomVar}{{\stackunder[1.0pt]{$a$}{\rule{0.8ex}{.075ex}}}_0}
\newcommand{\aRandomVec}{\bm{{\stackunder[1.0pt]{$a$}{\rule{0.8ex}{.075ex}}}}}

\newcommand{\Phik}{{\stackunder[1.0pt]{$\phi$}{\rule{0.8ex}{.075ex}}}_k}
\newcommand{\Phikun}{{\stackunder[1.0pt]{$\phi$}{\rule{0.8ex}{.075ex}}}_{k+1}}
\newcommand{\tauEpsilon}{{{\stackunder[1.0pt]{$\tau$}{\rule{0.6ex}{.075ex}}}_\epsilon}}
\newcommand{\Wtildek}{{{\stackunder[1.0pt]{$\tilde{w}$}{\rule{0.8ex}{.075ex}}}_k}}
\newcommand{\Wtildekun}{{{\stackunder[1.0pt]{$\tilde{w}$}{\rule{0.8ex}{.075ex}}}_{k+1}}}
\newcommand{\PhiZero}{{\stackunder[1.0pt]{$\phi$}{\rule{0.8ex}{.075ex}}}_0}

\newcommand{\Dzero}{{\stackunder[1.0pt]{$\delta$}{\rule{.8ex}{.075ex}}}_0}

\newcommand{\WtildeZero}{{{\stackunder[1.0pt]{$\tilde{w}$}{\rule{0.8ex}{.075ex}}}_0}}

\newcommand{\depsilon}{\delta_{\epsilon}}

\newcommand{\itksup}{\mathds{1}_{\{{{\stackunder[1.0pt]{$\tau$}{\rule{0.6ex}{.075ex}}}_\epsilon} > k\}}}
\newcommand{\kef}{\kappa_{ef}} 
\newcommand{\keg}{\kappa_{eg}} 
\newcommand{\Wtildekmun}{{{\stackunder[1.0pt]{$\tilde{w}$}{\rule{0.8ex}{.075ex}}}_{k-1}}}
\newcommand{\Phikmun}{{\stackunder[1.0pt]{$\phi$}{\rule{0.8ex}{.075ex}}}_{k-1}}

%% file: body.tex
\section{Introduction.}
Outstanding growth in the use of computers and sensors has 
attracted interest in myriad scientific and engineering fields 
for solving difficult optimization problems involving functions available only through a zeroth-order oracle (i.e., the functions are {\it black boxes}~\cite{AuHa2017}).
Derivative-free optimization (DFO~\cite{AuHa2017,CoScVibook,LMW2019AN}) addresses such situations where closed-form expressions and derivatives are not available.
This paper focuses on unconstrained stochastic DFO, wherein the objective function values are accessible only through a blackbox oracle corrupted by stochastic noise. We consider the problem 
\begin{equation}\label{probl1}
	\underset{\x\in\rn}{\min} f(\x), \qquad\mbox{with}\quad f(\x)=\Esp_{\ubar{\theta}}\left[f_{\ubar{\theta}}(\x)\right],
\end{equation}
where the values of the continuously differentiable function $f:\rn\to\R$ are available only via $f_{\ubar{\theta}}$, a stochastically noisy version of $f$, and where $\ubar{\theta}$ is a random variable whose distribution is possibly unknown. 
Stochastic gradient descent (SGD~\cite{RoMo1951}) is arguably the most used algorithm to solve~\eqref{probl1}. However, SGD is very sensitive to the choice of its learning rate; and not only does the method use step directions that are not necessarily descent ones~\cite{paquette2018stochastic}, but also it exhibits slow convergence~\cite{chen2018stochastic}. Similar remarks have motivated recent work in stochastic DFO, leading to many algorithms such as model-based methods~\cite{blanchet2016convergence,chen2018stochastic,shashaani2018astro}, those using a direct-search approach~\cite{anagnostidis2021direct,audet2019stomads,Ch2012,dzahini2020expected,dzahini2020constrained,rinaldi2022weak}, and  others~\cite{berahas2019global,RBSWlong21,miaoSchNeur2021,paquette2018stochastic} that employ stochastic estimation of the gradient $\nabla f(\x)$. Because of the unavailability of true function values and the resulting need to use stochastic estimates and/or models, the performance of these methods is highly dependent on how accurate those stochastic quantities are. For example,
STORM~\cite{blanchet2016convergence,chen2018stochastic} is a stochastic trust-region algorithm using so-called $\beta$-probabilistically $\ef$-accurate estimates of unknown function values as well as $\alpha$-probabilistically $\kappa$-fully linear models of the objective function. For convergence purposes, such estimates and models whose accuracies are dynamically controlled by means of a trust-region radius $\dk$ need to be sufficiently accurate probabilistically. More precisely, at a given iteration~$k$, the aforementioned estimates and models require, respectively, at least 
$\mathcal{O}\left(\dk^{-4}/(1-\sqrt{\beta})\right)$ 
and 
$\mathcal{O}\left((n+1)\max\accolade{\dk^{-2},\dk^{-4}}/(1-\alpha^{1/(n+1)})\right)$
function evaluations in a simple stochastic noise framework where no noisy gradient values are available; these quantities grow rapidly as $\alpha,\beta\in (0,1)$ or the dimension $n$ gets larger or when $\dk$ tends to zero. On the other hand, ASTRO-DF~\cite{shashaani2018astro} is a class of stochastic DFO trust-region algorithms where random estimates and polynomial interpolation models are constructed adaptively by using a Monte Carlo sampling whose extent is determined by continuously balancing and monitoring measures of sampling error and model bias. Recalling that the construction of {\it full space} quadratic interpolation models requires $q(n):=(n+1)(n+2)/2$ ($\ell(n):=n+1$ in the linear case) points~\cite{CoScVibook} and that the accuracy of Monte Carlo-based estimates crucially depends on the sampling size, as suggested by the strong law of large numbers~\cite[Theorem~2.1.8]{tao2012topicsInRMT}, ASTRO-DF also does not have low per-iteration costs in terms of function evaluations.

In light of these observations,
a key question naturally arises regarding model-based stochastic DFO methods: is there a way to improve the ability of these methods to handle large-scale problems. 
To answer a question similar to the one above in a context where the objective function is deterministic,  Cartis and  Roberts  \cite{CRsubspace2021} recently introduced RSDFO, a general framework of scalable subspace methods for model-based DFO. To achieve scalability, RSDFO approximates the deterministic objective only in subspaces using so-called $\Qk$-fully linear models, with~$\Qk\in \rnp$ denoting a matrix whose entries are randomly selected and where $p\leq n$ (ideally, $p\ll n$) is a user-determined parameter. The RSDFO framework was then specialized to deterministic nonlinear least-squares problems, and high-probability worst-case complexity bounds of the methods were derived under mild assumptions. As highlighted in~\cite{CRsubspace2021}, another model-based subspace DFO method with similarities to RSDFO but admitting no convergence analysis is the {\it moving ridge function} approach~\cite{GrossParks2021}, where an interpolation model is built in an {\it active space} that is determined by using existing objective function evaluations. Neumaier  et al.~\cite{NeFeSchLei2011} also proposed the VXQR method with no convergence analysis, where line searches are performed along directions selected in a subspace determined by previous iterates. More recently, a direct-search method based on probabilistic descent in {\it reduced spaces} was introduced in~\cite{RobRoy2022RedSpace} for the optimization of deterministic objective functions, where the polling directions are constructed in random subspaces; complexity bounds were also derived, making use of probabilistic properties related to both the latter directions and subspaces.

Of the cited methods, the only ones that achieve scalability using random subspace strategies are developed for deterministic objective functions; to our knowledge, no such method exists for stochastic DFO. We address this gap by introducing a stochastic trust-region algorithm in random subspaces (STARS), in which scalability arises from constructing and then minimizing stochastic models that approximate the objective function in low-dimensional random subspaces. By defining these random subspaces through user-provided random matrices $\QkRandom\in\rnp$ with $p\ll n$, only $q(p)$ and $\ell(p)$ points are required for the construction of quadratic and linear models, respectively, which are cheap models in terms of function evaluations since $p$ can always be chosen independently of $n$ in STARS.  While both the convergence and expected complexity analyses of STARS are inspired by those of STORM~\cite{blanchet2016convergence,chen2018stochastic}, dealing with the additional difficulties introduced by the randomness stemming from the entries of $\QkRandom$ is not a trivial task. As an example, no prior work exists showing how so-called $\beta_m$-{\it probabilistically} $(\kappa_{ef},\kappa_{eg};\Qk)$-{\it fully linear models} in random subspaces, crucial for the present analysis, can be made available. Moreover, the theoretical analyses of STORM utilizing random steps $\Sk\in\rn$ do not straightforwardly hold when the latter are replaced with  $\QkRandom\Sk\in\rn$ used by STARS. Our main results trivially imply those related to STORM in the full-space case where $\QkRandom=\bm{I_n}\in\rnn$ for all $k\in\N$, with $\bm{I_n}$
denoting the identity matrix. 

One of the key contributions of the present work is that it extends the analysis of a STORM-like framework~\cite{audet2019stomads,BerDioKunRoy2022LvnbrgMrqdt,blanchet2016convergence,chen2018stochastic,curtis2020adaptive,dzahini2020expected,dzahini2020constrained,paquette2018stochastic} to settings where additional randomness stems from internal mechanics (here, the entries of the random matrices~$\QkRandom$) of an algorithm for stochastic DFO. 
To our knowledge, STARS is the first DFO algorithm for stochastic objectives that achieves scalability using a random subspace strategy; STARS both extends STORM techniques to settings of scalable subspace methods and extends RSDFO techniques to stochastic objective functions. 
Another key contribution of this work is the results of Theorem~\ref{subspaceModel1} together with Corollary~\ref{CorQkFully}, which are, to the best of our knowledge, the first to rigorously provide a detailed strategy for the construction of $\beta_m$-probabilistically $(\kappa_{ef},\kappa_{eg};\Qk)$-fully linear models
for stochastic objective functions in random subspaces.
Moreover, relying on a framework introduced in~\cite{blanchet2016convergence} using results derived from martingale theory, the analysis of STARS demonstrates that while using random subspace models, the expected complexity of the algorithm is surprisingly similar to that of STORM~\cite{blanchet2016convergence}, stochastic direct-search~\cite{dzahini2020expected}, and line-search-based~\cite{paquette2018stochastic} methods up to constants.

The manuscript is organized as follows. Section~\ref{Section2} introduces the general framework of STARS and explains how it results in a stochastic process. Section~\ref{Section3} discusses various strategies for the selection of random subspaces, demonstrates how probabilistic estimates and models can be constructed in these subspaces, and provides conditions related to these random quantities that are necessary for the convergence of STARS. Sections~\ref{Section4} and~\ref{Section5},  respectively, present the convergence and expected complexity analyses of STARS. Section~\ref{Section6} presents numerical results, followed by a discussion and suggestions for future work.

\section{Random subspace trust-region and resulting stochastic process.}\label{Section2}
$ $
This section presents the general framework of STARS and explains how the proposed method results in a stochastic process.
\subsection{The random subspace stochastic trust-region method.}
Unlike the stochastic trust-region framework STORM~\cite{blanchet2016convergence,chen2018stochastic}, which builds full space random models $\ubar{m}_k(\x)$, $\x\in\rn$, to approximate $f(\x)$, STARS operates as follows. On iteration $k$, given a current iterate $\xk$, let $\mathcal{Y}_k\subseteq \rn$ be the affine space randomly chosen through the range of a matrix $\Qk\in\rnp$ 
whose entries are randomly selected; that is, 
\[\mathcal{Y}_k = \accolade{\xk+\Qk\bm{{\s}}:\bm{{\s}}\in\R^p}.\]
Then, given a trust-region radius $\dk>0$, a subspace model $\mhatk(\s), \s\in\rp$ is built only on $\mathcal{Y}_k$ using realizations of the stochastically noisy function $f_{\ubar{\theta}}$; the model  serves as an approximation of $f(\x)$ in $\mathcal{B}(\xk,\dk;\Qk):=\accolade{\xk+\Qk\s \in \mathcal{Y}_k:\normii{\s}\leq\dk}$. For concreteness, here we employ a quadratic subspace model given by 
\begin{equation*}
f(\xk+\Qk\s)\approx\mhatk(\bm{\s}):= f_k+\gkhat^\top\s+\s^\top\Hkhat\s,
\end{equation*}
where $\gkhat\in \R^p$ and $\Hkhat\in \Sym^{p}$ are the low-dimensional model gradient and Hessian, respectively. A {\it tentative step} $\Qk\sk\in\rn$ is produced by using the solution $\sk\in\rp$ obtained by approximately minimizing $\mhatk$ inside the trust  region; that  is, $\sk\approx{\argmin}\left\{ \mhatk(\s): \s\in\rp, \, \normii{\bm{{\s}}}\leq \dk\right\}$. Inspired by~\cite{CRsubspace2021,chen2018stochastic}, the {\it trial step} $\sk$ has to provide a sufficient decrease in $\mhatk$ by satisfying the following standard fraction of the Cauchy decrease condition.

\begin{assumption}
For every iteration $k$, a trial step $\sk\in\rp$ is computed so that\footnote{Throughout the manuscript, the matrix norm $\norme{\cdot}$ is supposed to be {\it consistent} with the Euclidean norm; that is, $\normii{\Ab\xb}\leq \norme{\Ab}\normii{\xb}$, which holds for both Frobenius and spectral matrix norms.}
\begin{equation}\label{CauchyDecrease}
\mhatk(\bm{0})-\mhatk(\bm{\sk})\geq \frac{\kappa_{fcd}}{2}\normii{\gkhat}\min\accolade{\dk,\frac{\normii{\gkhat}}{\max\left(\norme{\Hkhat},1\right)}},
\end{equation}
for some constant $\kappa_{fcd}\in (0,1]$.
\end{assumption}
Estimates $\fok$ and $\fsk$ of $f(\xk)$ and $f(\xk+\Qk\sk)$, respectively, are constructed by using evaluations of the noisy function $f_{\ubar{\theta}}$. The possible change in $f$ given by $\sk$ is measured by comparing the estimates $\fok$ and $\fsk$ through the value of a ratio $\rho_k$. An iteration is called successful if a decrease in the estimates is deemed sufficient, in which case the current solution is updated by $\xk+\Qk\sk$ and the trust-region radius is not decreased. Otherwise, the latter is decreased while the former is not updated. Because of inaccurate estimates, successful iterations qualified as {\it false} could lead to an increase in $f$. While the present algorithmic framework is not able to identify such iterations, as was the case in~\cite{audet2019stomads,chen2018stochastic,dzahini2020expected,dzahini2020constrained,paquette2018stochastic}, it will  be possible to prove later in the proof of Theorem~\ref{liminfTheorem} combined with~\eqref{randomWalk} that {\it true} iterations occur sufficiently often for convergence of the proposed method to hold.

\subsection{Stochastic process generated by the algorithm.}

The random variables considered here are all defined on the same probability space $(\Omega,\F, \pr)$, where $\Omega$, referred to as the {\it sample space}, is a nonempty set whose elements $\omega$ and subsets are called {\it sample points} and {\it events}, respectively. $\F$ is a collection of such events, which is called a {\it $\sigma$-algebra}. $\pr$ is a finite measure defined on the measurable space $(\Omega,\F)$, satisfying $\prob{\Omega}=1$ and referred to as the {\it probability measure}. An increasing subsequence $\accolade{\mathcal{S}_k}_k$ of $\sigma$-algebras of $\F$ will be called a {\it filtration}. Let $\mathscr{B}(\rn)$ be the $\sigma$-algebra generated by the open sets of $\rn$, also known as the Borel $\sigma$-algebra of $\rn$. A random variable $\ubar{z}$ is a measurable map defined on $(\Omega,\F, \pr)$ into the measurable space $(\rn, \mathscr{B}(\rn))$, where measurability means that $\accolade{\ubar{z}\in A}:=\accolade{\omega\in\Omega:\ubar{z}(\omega)\in A} =:{\ubar{z}}^{-1}(A)\in\F\ \forall A\in \mathscr{B}(\rn)$~\cite{bhattacharya2007basic}. For the remainder of the manuscript, vectors will be written in lowercase boldface (e.g., $\x\in\rn, n\geq 2$) while matrices will be written in uppercase boldface (e.g., $\Qk\in\rnp$), and underlined letters (e.g., $\ubar{z}, \ubar{\x}, \QkRandom$) will be used to denote random quantities; $z=\ubar{z}(\omega)$ will denote a realization of $\ubar{z}$.

\begin{algorithm}[htb]
\caption{STARS}
\label{algoStoScalTR} 
\textbf{[0] Initialization}\\
\hspace*{10mm}Choose constants $\gamma>1$, $\eta_1\in(0,1)$, $\eta_2>0$, $j_{\max}\in\N$, $\varepsilon\in(0,1)$, $\beta\in(0,1)$, \\ 
\hspace*{10mm}$c_1\geq 1$, initial trust-region radius $\delta_0>0$, and maximum trust-region radius\\
\hspace*{10mm}$\dmaxx=\gamma^{j_{\max}}\delta_0$, starting point $\bm{{\x}_0}\in\rn$, and dimension $p\in\intbracket{1,n}$.\\ 
\hspace*{10mm}Set the iteration counter $k \gets 0$.\\
\textbf{[1] Construction of subspace model}\\
\hspace*{10mm}Generate $\Qk$: a realization of a\footnotemark {\em WAM}($1-\varepsilon,\beta$), using a distribution $\mathbb{Q}_{k}$.\\
\hspace*{10mm}Build model $\mhatk:\rp\to\R$ that is $(\kef,\keg;\Qk)$-fully linear 
in  $\mathcal{B}(\xk,c_1\dk;\Qk)$.\\ 
\textbf{[2] Step calculation}\\
\hspace*{10mm}Compute $\sk\approx {\argmin}\left\{ \mhatk(\s) : \; \s\in\rp:\normii{\s}\leq \dk\right \}$ satisfying~\eqref{CauchyDecrease}.\\
\textbf{[3] Estimate computation}\\
\hspace*{10mm}Obtain estimates $\fok\approx f(\xk)$ and $\fsk\approx f(\xk+\Qk\sk)$ satisfying~\eqref{accEstimates}.\\
\textbf{[4] Updates}\\
\hspace*{10mm}Compute $\rho_k=\frac{\fok-\fsk}{\mhatk(\bm{0})-\mhatk(\sk)}.$\\
\hspace*{10mm}If $\rho_k\geq \eta_1$ and $\normii{\gkhat}\geq \eta_2\dk$ (\textbf{success}): \\
\hspace*{16mm}set $\bm{\xkun}=\xk+\Qk\sk$ and $\dkun=\min\accolade{\gamma\dk,\dmaxx}$. \\
\hspace*{10mm}Otherwise (\textbf{failure}): set $\bm{\xkun}=\xk$ and $\dkun=\gamma^{-1}\dk$.\\
\hspace*{10mm}Update the iteration counter $k \gets k+1$, and go to~\textbf{[1]}.\\
\end{algorithm}
\footnotetext{See Definition~\ref{wellAlignedDefinition} for the meaning of {\em WAM}($1-\varepsilon,\beta$).}

The deterministic models $\mhatk$ and the function estimates $\fok$ and $\fsk$ are constructed by using evaluations of the noisy function $f_{\ubar{\theta}}$ and the randomly selected entries of the deterministic matrix~$\Qk=\QkRandom(\omega)$.  Hence, $\mhatk$, $\fok$, and $\fsk$ can be considered as realizations of random models and estimates $\Mhatk$, $\Fok$, and $\Fsk$, respectively. Each iteration of Algorithm~\ref{algoStoScalTR} is therefore influenced by the behavior of these random quantities; and consequently, the algorithm results in a stochastic process. The present work shows that under certain conditions on the sequences~$\accoladekinN{\QkRandom}$, $\accoladekinN{\Fok,\Fsk}$, and $\accoladekinN{\Mhatk}$, the resulting stochastic process has desirable convergence properties, conditioned on the past, where {\it past} means the {\it past history} of the algorithm (see~Remark~\ref{whyTheseSigmaAlgebras}). In particular, for the needs of convergence and expected complexity analyses presented  in Sections~\ref{Section4} and~\ref{Section5}, and as will be seen  in Section~\ref{Section3}, the random models~$\Mhatk$ and estimates~$\Fok,\Fsk$ are required to be probabilistically sufficiently accurate conditioned on the past, as was the case in~\cite{chen2018stochastic}. Moreover, $\QkRandom$ must guarantee a high quality of the selected subspace probabilistically and conditioned on the past, as will be detailed in Section~\ref{Section31}. 

For all $k\in\N$, denote by $(\QkRandom)_{ij},  (i,j)\in \intbracket{1,n}\times\intbracket{1,p}=:\mathbb{I}_{n,p}$ the entries of the random matrix $\QkRandom$, where for $a\leq b$, $\intbracket{a,b}:=[a,b]\cap\mathbb{Z}$  throughout the manuscript. 
To formalize conditioning on the past, we consider the three filtrations $\accolade{\falgebra}_{k\in\N}$, $\accolade{\fqalgebra}_{k\in\N}$, and $\accolade{\fmqalgebra}_{k\in\N}$ defined respectively by 
\begin{equation*}
\begin{array}{rcl}
\falgebra &:=& \sigma\left(\Foell, \Fsell, \Mhatell,  (\QellRandom)_{ij} \mbox{ with}\  \ell \in \intbracket{0,k-1} \mbox{ \&}\ (i,j)\in \mathbb{I}_{n,p} \right)\\
\fqalgebra &:=& \sigma\left(\Foell, \Fsell, \Mhatell, (\QellRandom)_{ij}, (\QkRandom)_{ij} \mbox{ with}\  \ell \in \intbracket{0,k-1} \mbox{ \&}\ (i,j)\in \mathbb{I}_{n,p} \right)\\
\fmqalgebra &:=& \sigma\left(\Foell, \Fsell, \Mhatell, (\QellRandom)_{ij}, (\QkRandom)_{ij}, \ \Mhatk \mbox{ with}\  \ell \in \intbracket{0,k-1} \mbox{ \&}\ (i,j)\in \mathbb{I}_{n,p}  \right).
\end{array}
\end{equation*}
Here, for {\it completeness}~\cite{billingsley1995ThirdEdition}, $\mathcal{F}_{-1}=\sigma(\bm{\x_0})$.
By construction, $\falgebra\subset\fqalgebra\subset\fmqalgebra$, $\QkRandom$ is both~$\fqalgebra$ and~$\fmqalgebra$-measurable, and $\Mhatk$ is~$\fmqalgebra$-measurable, which imply that $\E{\QkRandom|\fqalgebra}=\E{\QkRandom|\fmqalgebra}=\QkRandom$ and $\E{\Mhatk|\fmqalgebra}=\Mhatk$. 

\begin{remark}\label{whyTheseSigmaAlgebras}
At a given iteration~$k$, when generating~$\QkRandom$, because no model or estimates were generated since the start of the iteration, it obviously follows from the construction of Algorithm~\ref{algoStoScalTR} that its past history can be formalized by~$\falgebra$. On the other hand, when constructing~$\Mhatk$, the past of the algorithm that includes~$\QkRandom$, given that the latter is already generated since the beginning of iteration~$k$, can therefore be formalized by~$\fqalgebra$. A similar observation  explains why~$\fmqalgebra$ can formalize the past history of Algorithm~\ref{algoStoScalTR} before the construction of~$\Fok$ and~$\Fsk$. These observations will also play an important role in the formalization of sufficient accuracy of the random models and estimates in Definitions~\ref{defProbQkfully} and~\ref{defProbAccEstim} and the formalization of high quality of the random subspaces in Assumption~\ref{wellAlignedAssumption}.  
\end{remark}

\section{Subspace selection and probabilistic models and estimates in random subspaces.}\label{Section3}
The random subspaces in which the models $\Mhatk$ are built are defined by specific random matrices. For efficiency of the proposed method and for convergence needs, not only must these subspaces be probabilistically rich enough, but also the accuracies of the models and estimates of unknown function values must hold with sufficiently high, but fixed, probabilities. This section discusses strategies for such random subspace selection, demonstrates by means of rigorous results how probabilistic estimates and subspace models can be constructed, and provides conditions that are necessary for the convergence of STARS.

\subsection{Subspace selection strategies.}\label{Section31}
To choose a subspace, 
there must be enough gradient living in the subspace to ensure analogous reduction of $f$. In other words, the matrix $\Qk\in\rnp$ determining the subspace needs to satisfy (probabilistically) $\normii{\Qk^\top\nabla f(\xk)}\geq \alpha_{Q}\normii{\nabla f(\xk)}$ at every iteration, for some constant $\alpha_{Q}\in (0,1)$ independent of $k$. This requirement motivates the following modified definition suggested by~\cite{CRsubspace2021,shao2022Thesis}, which basically says that at least some fraction of the gradient of~$f$ must be maintained after projecting it into the selected subspace.

\begin{definition}\label{wellAlignedDeterministic}
For $\alpha_{Q}\in (0,1)$, a matrix $\Q$ is $\alpha_{Q}$-well aligned if, for any vector $\bm{{\vi}}\in\rn$, $\normii{\Q^\top\bm{{\vi}}}\geq \alpha_{Q}\normii{\bm{{\vi}}}$.
\end{definition}

Recalling Remark~\ref{whyTheseSigmaAlgebras} and inspired by~\cite{CRsubspace2021}, we propose a probabilistic quantification of the quality of the subspace selection as follows.

\begin{definition}\label{wellAlignedDefinition}
For fixed $\alpha_Q,\beta_Q\in (0,1)$, a sequence $\accoladekinN{\QkRandom}$ of random matrices 
is $(1-\beta_Q)$-probabilistically $\alpha_{Q}$-well aligned if, for any $\F_{k-1}$-measurable random vector $\viRandom$ with realizations $\vi\in \rn$, the events 
$\mathcal{A}_k:= \accolade{\normii{\QkRandomTranspose\viRandom}\geq \alpha_{{Q}}\normii{\viRandom}}$ satisfy 
\begin{equation}\label{Akineq}
\prob{\mathcal{A}_k|\F_{k-1}}=\E{\iak|\F_{k-1}}\geq 1-\beta_Q,
\end{equation}
where $\iak$ denotes the indicator function of the event $\mathcal{A}_k$; that  is, $\iak(\omega)=1$ if $\omega\in \mathcal{A}_k$ and $\iak(\omega)=0$ otherwise. Random matrices satisfying~\eqref{Akineq} will be referred to as {WAM}($\alpha_Q,\beta_Q$).
\end{definition}
For convergence purposes  in Sections~\ref{Section4} and~\ref{Section5}, the following will be assumed throughout the manuscript.
\begin{assumption}\label{wellAlignedAssumption}
The sequence $\accoladekinN{\QkRandom}$ of matrices used by Algorithm~\ref{algoStoScalTR} is $(1-\beta_Q)$-probabilistically $\alpha_{Q}$-well aligned for some fixed $\alpha_Q\in (0,1)$ and $\beta_Q\in (0,1/2)$.
\end{assumption}



The next result from~\cite[Lemma~1]{LauBeaMaP2000} provides a slight generalization of an exponential inequality for chi-square distributions. It will be required later for the proof of Theorem~\ref{GaussianJLTmatrix}, inspired by~\cite[Theorem~2.1]{WoodrDavP2014}, providing a random {\it matrix ensemble} that will be shown in Corollary~\ref{JLCorollary1} to satisfy Assumption~\ref{wellAlignedAssumption}.
\begin{lemma}\label{BeatriceLaurent}
Let $(\yOne,\yTwo,\dots,\ypi)$ be i.i.d.\ Gaussian random variables with mean zero and variance one. Let $a_1$, $a_2$, $\dots$, $a_p$ be nonnegative real numbers. Let $\ubar{z}$ be the random variable defined by $\ubar{z}=\sum_{i=1}^{p}a_i(\yii^2-1)$. Then it holds that for any $t\geq 0$,
\[\prob{\ubar{z}\leq -2\normii{\bm{a}}\sqrt{t}}\leq e^{-t}.\]
\end{lemma}

\begin{theorem}\label{GaussianJLTmatrix}
Let $\varepsilon>0$, $\beta<1$, $p\geq 4\varepsilon^{-2}\log(1/\beta)$, and $\RMatrix=\left({\RrandomVar}_{ij}\right)_{\underset{1\leq j\leq n}{1\leq i\leq p}}\in \rpn$ be a random matrix with i.i.d.\ standard Gaussian entries. Let $\SMatrix=\left(\SrandomVar_{ij}\right)_{\underset{1\leq j\leq n}{1\leq i\leq p}}\in \rpn$ be a random matrix defined by $\SMatrix=\frac{1}{\sqrt{p}}\RMatrix$; that is, $\SrandomVar_{ij}=\frac{1}{\sqrt{p}}\RrandomVar_{ij}\sim\n(0,1/p)$. Then for all $\vi\in\rn$,
$\pr\left[\normii{\SMatrix\vi}^2\geq (1-\varepsilon)\normii{\vi}^2\right]\geq 1-\beta$.
\end{theorem}

\begin{proof}
The proof is inspired by that of~\cite[Lemma~2.12]{WoodrDavP2014}. Consider any deterministic vector\\ $\vi=\left(v_1,v_2,\dots,v_n\right)^\top\in\rn$,  and denote by $\sRandom_{: j}\in\R^p, \, j \in\intbracket{1,n}$ the columns of the random matrix $\SMatrix$. Then the $i$th component of the $p$-tuple $\SMatrix\vi=\sum_{j=1}^{n}v_j\sRandom_{: j}$ is the Gaussian random variable $(\SMatrix\vi)_i=\sum_{j=1}^{n}v_j{\SrandomVar}_{i j}$ with mean zero and variance $\V\left[(\SMatrix\vi)_i\right]= \sum_{j=1}^{n}v_j^2\var{{\SrandomVar}_{i j}}=\frac{\normii{\vi}^2}{p}$. It follows from the independence of $(\SMatrix\vi)_i, \, i\in\intbracket{1,p}$ that the random variable $\normii{\SMatrix\vi}^2=\sum_{i=1}^p(\SMatrix\vi)_i^2$ is equal in distribution to $\frac{\normii{\vi}^2}{p}{\ubar{w}}$, where ${\ubar{w}}$ is a $\chi^2_p$ random variable with $p$ degrees of freedom. Applying Lemma~\ref{BeatriceLaurent} with $t=\frac{\varepsilon^2p}{4}$ and $a_i=1,\ i\in\intbracket{1,p}$, yields $\prob{{\ubar{w}}-p\leq -\varepsilon p}\leq e^{-\frac{\varepsilon^2p}{4}}.$ 
Hence, $\pr\left[\normii{\SMatrix\vi}^2\geq (1-\varepsilon)\normii{\vi}^2\right]=1-\pr\left[{\ubar{w}}\leq(1-\varepsilon)p\right]\geq 1-e^{-\frac{\varepsilon^2p}{4}}$.
The proof is completed by noticing that $1-2e^{-\frac{\varepsilon^2p}{4}}\geq 1-\beta$ if $p\geq 4\varepsilon^{-2}\log(1/\beta)$.
\end{proof}

The next corollary shows that Assumption~\ref{wellAlignedAssumption} can be satisfied by using matrices resulting from Theorem~\ref{GaussianJLTmatrix}, and whose number~$p$ of columns does not depend on~$n$.
\begin{corollary}\label{JLCorollary1}
At a given iteration $k$, let the subspace selection matrix $\QkRandomTranspose=\SMatrix\in \rpn$ be provided by Theorem~\ref{GaussianJLTmatrix}, with $\varepsilon=1-\alpha_Q$ and $\beta=\beta_Q$, for some $\alpha_Q\in (0,1)$ and $\beta_Q\in (0,1/2)$. Assume that $\QkRandom$ is independent of $\QzeroRandom, \QOneRandom,\dots, \QkmunRandom$ and ${\ubar{\theta}}$. Then Assumption~\ref{wellAlignedAssumption} holds; that is, for any $\viRandom$ $\falgebra$-measurable, 
\begin{equation*}
\prob{\accolade{\normii{\QkRandomTranspose\viRandom}\geq \alpha_{{Q}}\normii{\viRandom}}|\F_{k-1}}\geq 1-\beta_Q.
\end{equation*}
\end{corollary}
\begin{proof} Since $\QkRandom$ is independent of $\falgebra$ and $\viRandom$ is $\falgebra$-measurable, it suffices to show that\\ $\prob{\accolade{\normii{\QkRandomTranspose\vi}\geq \alpha_{{Q}}\normii{\vi}}}\geq 1-\beta_Q$
for any deterministic vector $\vi\in\rn$, which easily follows from Theorem~\ref{GaussianJLTmatrix} and the inclusion $\accolade{\normii{\SMatrix\vi}^2\geq (1-\varepsilon)\normii{\vi}^2} \subseteq \accolade{\normii{\SMatrix\vi}\geq (1-\varepsilon)\normii{\vi}}$ due to the inequality $1-\varepsilon\geq (1-\varepsilon)^2$.
\end{proof}


Inspired by~\cite{CRsubspace2021}, a technical assumption needed for convergence and expected complexity analyses later in Sections~\ref{Section4} and~\ref{Section5} is the following. 
\begin{assumption}\label{QmaxAssumption}
Any realization $\Qk$ of the random matrix used by Algorithm~\ref{algoStoScalTR} satisfies $\norme{\Qk}\leq \Qmax$ for all $k\in\N$, for some constant $\Qmax >0$ independent of~$k$.
\end{assumption}
The next theorem, partially proved in the Appendix and formulated based on results from~\cite{KaNel2014SparseLidenstrauss}, provides another technique for constructing {\em WAM}($1-\varepsilon,\beta$) satisfying Assumption~\ref{wellAlignedAssumption} through so-called Johnson--Lindenstrauss (JL) transforms~\cite{JohnsonLind1984,KaNel2014SparseLidenstrauss}, and whose number~$p$ of columns does not depend on~$n$. Moreover, one can easily see that the resulting matrix ensemble satisfies Assumption~\ref{QmaxAssumption} with $\Qmax=\sqrt{n}$, unlike the one from Theorem~\ref{GaussianJLTmatrix} where $\norme{\QkRandom}\leq \Qmax$ holds with high probability with\footnote{The reader is referred to~\cite{knuth1976big} for details regarding the notation~$\Theta(\cdot)$.} $\Qmax=\Theta(\sqrt{n/p})$~\cite[Corollary~3.11]{bandeira2016sharp}.
\begin{theorem}\label{sHashingTheorem}
Let $\varepsilon>0$, $\beta<1/2$, and $\SMatrix=\left(\SrandomVar_{ij}\right)_{\underset{1\leq j\leq n}{1\leq i\leq p}}\in \rpn$ be a real-valued random matrix defined by
$\SrandomVar_{ij}=\frac{1}{\sqrt{r}}\etarandomVar_{ij}\sigmarandomVar_{ij},$
where $r\in\N\setminus\accolade{0}$ and $\sigmarandomVar_{ij}$ are independent Rademacher random variables. The $\etarandomVar_{ij}$ satisfying $\prob{\sum_{i=1}^p\etarandomVar_{ij}=r}=1$ for all $j\in\intbracket{1,n}$ are negatively correlated indicator random variables for the events $\accolade{\SrandomVar_{ij}\neq 0}$; that  is,  they satisfy
$\E{\underset{(i,j)\in T}{\prod}\etarandomVar_{ij}}\leq (r/p)^{\abs{T}}$ 
for all $T\subseteq\intbracket{1,p}\times\intbracket{1,n}$ with $\abs{T}\leq\ell$, where $\ell\geq \ln\left[\beta^{-1}(\ell+1)\ell/2\right]$. Then $\pr\left[(1-\varepsilon)\normii{\vi}\leq\normii{\SMatrix\vi}\leq (1+\varepsilon)\normii{\vi}\right] > 1-\beta$ for any $\vi\in\rn$, provided $r\geq 8e^4\sqrt{e}(\ell+1)/(2\varepsilon-\varepsilon^2)$ and $p=2r^2/(e\ell)$.
\end{theorem}

The next corollary shows that {\em WAM}($\alpha_Q,\beta_Q$) satisfying Assumptions~\ref{wellAlignedAssumption} and~\ref{QmaxAssumption} can be provided by specific JL transforms $\SMatrix$ resulting from Theorem~\ref{sHashingTheorem}.
\begin{corollary}
At a given iteration $k$, let the subspace selection matrix $\QkRandomTranspose=\SMatrix\in \rpn$ be provided by Theorem~\ref{sHashingTheorem}, with $\varepsilon=1-\alpha_Q$ and $\beta=\beta_Q$. Assume that $\QkRandom$ is independent of $\QzeroRandom, \QOneRandom,\dots, \QkmunRandom$ and~${\ubar{\theta}}$. Then Assumption~\ref{wellAlignedAssumption} holds.
\end{corollary}
\begin{proof}
As was explained in the proof of Corollary~\ref{JLCorollary1}, it suffices to show that for all deterministic vectors $\vi\in\rn$, $\prob{\accolade{\normii{\QkRandomTranspose\vi}\geq \alpha_{{Q}}\normii{\vi}}}\geq 1-\beta_Q$. This immediately follows from the result of Theorem~\ref{sHashingTheorem}, together with the inclusion
$\accolade{(1-\varepsilon)\normii{\vi}\leq\normii{\SMatrix\vi}\leq (1+\varepsilon)\normii{\vi}} \subseteq \accolade{\normii{\SMatrix\vi}\geq (1-\varepsilon)\normii{\vi}}.$
\end{proof}

\subsection{Probabilistic models and estimates in random subspaces.}
As mentioned at the beginning of Section~\ref{Section3}, the models used by Algorithm~\ref{algoStoScalTR} need to be sufficiently accurate. Motivated by~\cite{chen2018stochastic}, this sufficient accuracy is formalized for deterministic subspace models by the following measure of accuracy inspired by~\cite{CRsubspace2021}.
\begin{definition}\label{QkFullyLinearModels}
Assume $\nabla f$ is  Lipschitz continuous. Given $\Qk\in\rnp$ and $\xk\in\rn$, a function $\mhatk:\rp\to\R$ is a $(\kef,\keg; \Qk)$-fully linear model of $f$ in $\mathcal{B}(\xk,\dk;\Qk)$ for some $\dk>0$ if there exist constants $\kappa_{ef}, \kappa_{eg} >0$ independent of $k$ such that for all $\s\in\rp$ with $\normii{\s}\leq \dk$,
\begin{equation}\label{Eq36}
\abs{f(\xk +\Qk\s)-\mhatk(\s)}\leq \kappa_{ef}\dk^2\quad \mbox{and}\quad
\normii{\Qk^{\top}\nabla f(\xk +\Qk\s)-\nabla\mhatk(\s)}\leq \kappa_{eg}\dk.
\end{equation}
\end{definition}
As pointed out in~\cite{CRsubspace2021}, the gradient condition in~\eqref{Eq36} is due to the fact that given any fixed~$\x$ and~$\Q$, if $\hat{f}(\s):=f(\x+\Q\s)$, then $\nabla_{\s} \hat{f}(\s)=\Q^{\top}\nabla_{\x} f(\x+\Q\s)$. Moreover, denoting by $\bm{I}$ the identity matrix in full-dimensional subspaces where $p=n$, then $(\kef,\keg; \bm{I})$-fully linear models correspond to standard fully linear models~\cite[Definition~6.1]{CoScVibook}.

Recalling Remark~\ref{whyTheseSigmaAlgebras}, the next definition formalizing sufficient accuracy of probabilistic subspace models $\Mhatk$ is a stochastic variant of Definition~\ref{QkFullyLinearModels}, which generalizes~\cite[Definition~3.4]{chen2018stochastic} to low-dimensional subspaces. The existence of such models will be rigorously demonstrated  in Theorem~\ref{subspaceModel1} and Corollary~\ref{CorQkFully}.
\begin{definition}\label{defProbQkfully}
A sequence $\accolade{\Mhatk}$ of random models is $\beta_m$-probabilistically \\ $(\kef,\keg; \Qk)$-fully linear with respect to the random sequence $\accolade{\mathcal{B}(\Xk,\Dk;\QkRandom)}$ if the events
\[I_k^Q=\accolade{\mbox{Given}\ \Qk\ \mbox{and}\ \xk, \Mhatk\ \mbox{is}\ (\kef,\keg;\Qk)\mbox{-fully linear for}\ f\ \mbox{in}\ \mathcal{B}(\Xk,\Dk;\QkRandom)}\]
satisfy the supermartingale-like condition
\begin{equation}\label{ikQProb}
\prob{I^Q_k|\fqalgebra}=\E{\iik|\fqalgebra}\geq \beta_m.
\end{equation}
\end{definition}
Recalling that estimates of unknown function values also need to be sufficiently accurate, we introduce a formalization of sufficient accuracy inspired by~\cite{audet2019stomads,chen2018stochastic,dzahini2020expected,dzahini2020constrained,paquette2018stochastic} and motivated by~\cite{CRsubspace2021} as follows.
\begin{definition}
Given $\ef>0$ and $\Qk\sk\in\rn, \, \xk\in\rn$, $\fok$ and $\fsk$ are called $\ef$-accurate estimates of $f(\xk)$ and $f(\xk+\Qk\sk)$, respectively, for a given $\dk$ if
\begin{equation}\label{accEstimates}
\abs{\fok-f(\xk)}\leq\ef\dk^2\quad \mbox{ and }\quad  \abs{\fsk-f(\xk+\Qk\sk)}\leq\ef\dk^2.
\end{equation}
\end{definition}

Recalling Remark~\ref{whyTheseSigmaAlgebras},
we  extend  this definition to the following stochastic variant.
\begin{definition}\label{defProbAccEstim}
A sequence of random estimates $\accolade{\Fok,\Fsk}$ is said to be $\beta_f$-probabilistically $\ef$-accurate with respect to the corresponding sequence $\accolade{\Xk,\Dk,\QkRandom\Sk}$ if the events 
\begin{eqnarray*}
J_k^Q &=&\left\lbrace\mbox{Given}\ \Qk\ \mbox{and}\ \xk, \Fok \mbox{ and } \Fsk \mbox{ are }\ef\mbox{-accurate estimates of } f(\xk)  \mbox{ and}\right.\\
& &\qquad\qquad\qquad\qquad\qquad\qquad\qquad\qquad \left.  f(\xk+\Qk\sk), \mbox{ respectively, for } \Dk \right\rbrace
\end{eqnarray*}
satisfy the supermartingale-like condition
$\prob{J^Q_k|\fmqalgebra}\geq \beta_f.$
\end{definition}

Based on Definitions~\ref{wellAlignedDefinition}, \ref{defProbQkfully}, and~\ref{defProbAccEstim}, {\it true} iterations, in other words, those for which $\iak\iik\ijk=1$, occur with a total probability of at least $\tilde{\beta}:=\beta_f\beta_m(1-\beta_Q)$ conditioned on $\falgebra$. Noting that by construction $\iak$ is both $\fqalgebra$-measurable and $\fmqalgebra$-measurable since $\fqalgebra\subset\fmqalgebra$, and that $\iik$ is $\fmqalgebra$-measurable, then for $F:= \accolade{\mathcal{A}_k\cap I^Q_k\cap J_k^Q}$, we have
\begin{equation}\label{prodBetaIneq}
\begin{split}
\pr(F|\falgebra)&=\E{\iak\iik\ijk|\falgebra}
=\E{\iak\iik\E{\ijk|\fmqalgebra} |\falgebra}\\
&\geq \beta_f\E{\iak\E{\iik|\fqalgebra} |\falgebra}\geq \beta_f\beta_m\E{\iak|\falgebra}\geq \tilde{\beta}.
\end{split}
\end{equation}
The second equality and the first two inequalities used the  {\it William's Tower Property} for conditional expectations~\cite[Theorem~34.4.]{billingsley1995ThirdEdition}, namely, the fact that if ${\ubar{z}}$ is integrable and the $\sigma$-algebras $\mathcal{G}_1$ and $\mathcal{G}_2$ satisfy $\mathcal{G}_1\subset\mathcal{G}_2$, then $\E{{\ubar{z}}|\mathcal{G}_1}=\E{\E{{\ubar{z}}|\mathcal{G}_2}|\mathcal{G}_1}$. Moreover, for a $\sigma$-algebra $\mathcal{G}$,  $\E{{\ubar{r}}{\ubar{u}}|\mathcal{G}}={\ubar{r}}\E{{\ubar{u}}|\mathcal{G}}$ if~${\ubar{r}}$ is $\mathcal{G}$-measurable and~${\ubar{r}}$ and~${\ubar{r}}{\ubar{u}}$ are integrable~\cite[Theorem~34.3.]{billingsley1995ThirdEdition}.

Note that even though the present algorithmic framework does not distinguish true iterations from {\it false} ones, it holds that $\prob{\underset{k\to\infty}{\limsup}\ \ubar{w}_k=\infty}=1$, where
\begin{equation}\label{randomWalk}
\ubar{w}_k:=\sum_{i=0}^k\left(2\mathds{1}_{\mathcal{A}_i}\mathds{1}_{I^Q_i}\mathds{1}_{J^Q_i}-1\right),
\end{equation}
as will be seen later in the proof of Theorem~\ref{liminfTheorem} provided $\tilde{\beta}>1/2$. This shows that true iterations occur sufficiently often for the convergence of Algorithm~\ref{algoStoScalTR} to hold (see, e.g.,~\cite[Theorem~3.6]{dzahini2020constrained} for a similar result).

\begin{remark}
Calculations similar to those in~\eqref{prodBetaIneq} easily yield 
\begin{eqnarray}\label{prodBetafBetam}
\prob{I^Q_k\cap J_k^Q|\fqalgebra}\geq \beta_f\beta_m, & & \prob{J_k^Q|\fqalgebra}\geq \beta_f,\\
\label{otherProbIneqs}
\prob{I^Q_k|\falgebra}\geq\beta_m & \qquad \mbox{ and } \qquad & \prob{J^Q_k|\falgebra}\geq\beta_f,
\end{eqnarray}
which will be useful in the proofs of Theorems~\ref{zerothOrderScal} and~\ref{liminfTheorem}, respectively. 
\end{remark}

\subsection{Construction of probabilistic estimates and models in random subspaces.}
The construction of probabilistic estimates of Definition~\ref{defProbAccEstim} trivially follows the strategies described in~\cite[Section~2.3]{audet2019stomads}, \cite[Section~5]{chen2018stochastic}, and~\cite[Section~5.1]{dzahini2020constrained} and hence are not presented here again.

Next is stated an extension to subspaces of~\cite[Lemma~9.4]{AuHa2017}, which will be useful for the proof of Theorem~\ref{subspaceModel1}, one of the main results of the present work. Its proof is presented in the Appendix.

\begin{lemma}\label{FundTheorCalc}
Let $f$ be differentiable with a $\Lg$-Lipschitz continuous gradient, and let $\s,\di\in\rp$. For a given matrix $\Qk$ and vector $\xk$, let $\hat{f}(\s)=f(\xk+\Qk\s)$, and assume that Assumption~\ref{QmaxAssumption} holds. Then 
\[\abs{\hat{f}(\s+\di)-\hat{f}(\s)-\di^{\top}\nabla \hat{f}(\s)}\leq \frac{1}{2}\Lg\Qmax^2\normii{\di}^2.\]
\end{lemma}
The next result, which is a stochastic variant of~\cite[Theorem~9.5]{AuHa2017}, shows, together with Corollary~\ref{CorQkFully}, how $\beta_m$-probabilistically $(\kef,\keg; \Qk)$-fully linear models of Definition~\ref{defProbQkfully} can be made available. 
\begin{theorem}\label{subspaceModel1} 
Let the assumptions of Lemma~\ref{FundTheorCalc} hold.
Assume that there exists a finite constant $V_f>0$ such that 
$\V_{\ubar{\theta}}[f_{\ubar{\theta}}(\x)]\leq V_f$ for all $\x\in\rn$. For all $i=0,1,\dots, p$, let ${\ubar{\theta}}^i_\ell, \ell=1,\dots,\pi_k$ be independent random samples of the independent random variables ${\ubar{\theta}}^i$ following the same distribution as ${\ubar{\theta}}$, and define $\Fsk(\sii,{\ubar{\theta}}^i):=\frac{1}{\pi_k}\sum_{\ell=1}^{\pi_k} f_{{\ubar{\theta}}^i_\ell}(\xk+\Qk\sii)$, where the points  $\SSk:=\accolade{\szero,\sone,\dots,\spi}\subseteq\rp\cap\mathcal{B}(\bm{0},c_1\dk)$ are affinely independent with $\szero=\bm{0}$ and $c_1\geq1$ a~constant. Let $\rho_k:=\underset{\sii \in \SSk}{\max}\normii{\sii-\szero}$ denote the approximate diameter of $\SSk$ and
define $\Lkhat:=\frac{1}{\rho_k}\left[\sone-\szero\quad  \cdots\quad \spi-\szero \right]\in\rpp$. Let $\bm{1}:=[1,\dots,1]^{\top}\in\R^{p+1}$, and define $\bm{Y}:=\left[\szero\ \, \sone\ \cdots\ \spi\right]\in\R^{p\times (p+1)}$. Consider the random vector  $\ubar{f}_k(\SSk,{\ubar{\theta}}):=\left[\Fsk(\szero,{\ubar{\theta}}^0), \Fsk(\sone,{\ubar{\theta}}^1),\dots, \Fsk(\spi,{\ubar{\theta}}^p)\right]^{\top}$ $\in\R^{p+1}$ and the random linear model $\Mhatk(\s):=\azeroRandomVar+\aRandomVec^{\top}\s$, where $(\azeroRandomVar,\aRandomVec)\in\R\times\rp$ is the unique solution of $\left[\bm{1}\quad \bm{Y}^{\top}\right]\comb{\azeroRandomVar}{\aRandomVec}=\ubar{f}_k(\SSk,{\ubar{\theta}})$. 
Defining\footnote{While $\kef$ and $\keg$ seem to always depend on $k$, this is not the case since $\norme{\Lkhat^{-1}}$ can be controlled by the geometry of the set $\SSk$, as will be seen later by means of Corollary~\ref{CorQkFully}.}
\begin{equation}\label{kefkeg}
\kef:= \frac{1}{2}\left(1+\frac{1}{2}\textcolor{black}{c_1^2\delta_{\max}}+2\textcolor{black}{c_1}\sqrt{p}\norme{\Lkhat^{-1}}\right)\Lg\Qmax^2,\quad \quad 
\keg:= \left(1+\textcolor{black}{c_1}\sqrt{p}\norme{\Lkhat^{-1}}\right)\Lg\Qmax^2,
\end{equation}
consider the events 
\begin{equation*}
E_f:= \accolade{\abs{\hat{f}(\s)-\Mhatk(\s)}\leq \kef\dk^2}, \qquad  E_g:=\accolade{\normii{\nabla \hat{f}(\s)-\nabla\Mhatk(\s)}\leq \keg\dk}.
\end{equation*}
Then 
\begin{equation}\label{gradientBoundQlinear}
\pr\left[E_f\cap E_g\right]\geq \beta_m \quad \mbox{for some}\ \beta_m\in(0,1)\ \mbox{ and all }\ \normii{\s}\leq \dk,
\end{equation}
provided 
\begin{equation}\label{ppkChoice}
\pi_k\geq \frac{16V_f}{\textcolor{black}{c_1^2}\Lg^2\Qmax^4\rho_k^2\min\{\dk^2,\dk^4\}\left(1-\beta_m^{1/(p+1)}\right)}.
\end{equation}
\end{theorem}
\begin{proof}
We build the random linear model $\Mhatk(\s)=\azeroRandomVar+\aRandomVec^{\top}\s$ by seeking values for $({\azeroRandomVar},\aRandomVec)\in\R\times \R^{p}$ such that 
$\Mhatk(\sii)=\Fsk(\sii,{\ubar{\theta}}^i)$  for all $\sii$ in the interpolation set $\SSk$,
which is equivalent to the $(p+1)\times(p+1)$ linear system of equations
\begin{equation}\label{systemTwo}
\left[\bm{1}\quad \bm{Y}^{\top}\right]\comb{\azeroRandomVar}{\aRandomVec}=\ubar{f}_k(\SSk,{\ubar{\theta}}).
\end{equation}
Since the points in $\SSk$ are affinely independent, the matrix $\Lkhat$ is invertible (see, e.g., \cite[Proposition~9.1]{AuHa2017}). Hence, system~\eqref{systemTwo} has a unique solution. Moreover, $\aRandomVec=\nabla\Mhatk(\s)$ can also be computed by solving the $p\times p$ linear system 
\begin{equation}\label{systemThree}
\rho_k \Lkhat^{\top}\aRandomVec=\bm{{\ubar{\delta}}^{f_k(\SSk,{\ubar{\theta}})}}, \quad\mbox{where}\quad 
\bm{{\ubar{\delta}}^{f_k(\SSk,{\ubar{\theta}})}}={\begin{bmatrix}
		{\Fsk(\sone,{\ubar{\theta}}^1)-\Fsk(\szero,{\ubar{\theta}}^0)}\\
		{\Fsk(\stwo,{\ubar{\theta}}^2)-\Fsk(\szero,{\ubar{\theta}}^0)}\\
		{\vdots}\\
		{\Fsk(\spi,{\ubar{\theta}}^p)-\Fsk(\szero,{\ubar{\theta}}^0)}
\end{bmatrix}}\in \rp.
\end{equation}

To prove~\eqref{gradientBoundQlinear}, we consider the events
\[B_i:=\accolade{\abs{\Fsk(\sii,{\ubar{\theta}}^i)-\hat{f}(\sii)}\leq \frac{1}{2}\kappa_{eg}'\rho_k\min\{\dk,\dk^2\}},\quad i=0,1,\dots,p,\]
where $\kappa_{eg}':=\frac{1}{2}\textcolor{black}{c_1}\Lg\Qmax^2$; and we assume that 
\begin{equation}\label{BiProb}
\prob{B_i}\geq \beta_m^{1/(p+1)}\quad\mbox{for some }\ \beta_m\in (0,1)\  \mbox{and all}\ i=0,1,\dots,p.
\end{equation}
It follows from the independence of the random variables $\Fsk(\sii,{\ubar{\theta}}^i),\ i=0,1,\dots,p$, that the event $B:=\bigcap_{i=0}^pB_i$ satisfies $\prob{B}=\prod_{i=0}^{p}\prob{B_i}\geq \beta_m$. Then the remainder of the proof considers three  parts. The first two show respectively that $B\subseteq E_g$ and $B\subseteq E_f$, which imply that $B\subseteq E_f\cap E_g$ and hence $\prob{E_f\cap E_g}\geq \beta_m$. Part~3 provides the condition under which~\eqref{BiProb} holds.\\
$ $\\
\textbf{Part~1 ($B\subseteq E_g$)}. To demonstrate the latter inclusion, we will show that 
\begin{equation}\label{FirstEvent}
B\subseteq E_1:= \accolade{\normii{\Lkhat^{\top}\left[\aRandomVec-\nabla \hat{f}(\szero)\right]}\leq 2\sqrt{p}\kappa_{eg}'\dk}.
\end{equation}

Then, since the matrix norm is consistent with the Euclidean norm, the inequality \[\normii{\Lkhat^{-\top}\Lkhat\left(\aRandomVec-\nabla \hat{f}(\szero)\right)}\leq \norme{\Lkhat^{-\top}}\normii{\Lkhat\left(\aRandomVec-\nabla \hat{f}(\szero)\right)} \]
implies that $E_1\subseteq E_2:=\accolade{\normii{\aRandomVec-\nabla \hat{f}(\szero)}\leq 2\sqrt{p}\kappa_{eg}'\norme{\Lkhat^{-1}}\dk}$, where we used the fact that $\norme{\Lkhat^{-\top}}=\norme{\Lkhat^{-1}}$. Using the $\Lg$-Lipschitz continuity of $\nabla f$ and the fact that $\nabla \hat{f}(\s)=\Qk^{\top}\nabla f(\xk+\Qk\s)$ with $\norme{\Qk}\leq \Qmax$, we get 
\begin{eqnarray*}
\normii{\nabla \hat{f}(\s)-\aRandomVec}&\leq&\normii{\aRandomVec-\nabla \hat{f}(\szero)}+\normii{\nabla \hat{f}(\szero)-\nabla \hat{f}(\s)}\\
&\leq& 2\sqrt{p}\kappa_{eg}'\norme{\Lkhat^{-1}}\dk + \Lg\norme{\Qk^{\top}}\norme{\Qk}\normii{\szero-\s}\\
&\leq& \left(2\sqrt{p}\kappa_{eg}'\norme{\Lkhat^{-1}}+\Lg\Qmax^2\right)\dk=\left(1+\textcolor{black}{c_1}\sqrt{p}\norme{\Lkhat^{-1}}\right)\Lg\Qmax^2\dk,
\end{eqnarray*}
which implies that $E_2\subseteq E_g:=\accolade{\normii{\nabla \hat{f}(\s)-\aRandomVec}\leq\left(1+\textcolor{black}{c_1}\sqrt{p}\norme{\Lkhat^{-1}}\right)\Lg\Qmax^2\dk}$.

To show~\eqref{FirstEvent}, we first notice, using~\eqref{systemThree}, that $\Lkhat^{\top}\aRandomVec=\frac{1}{\rho_k}\bm{{\ubar{\delta}}^{f_k(\SSk,{\ubar{\theta}})}}$, and then 
\begin{equation*}
\Lkhat^{\top}\nabla \hat{f}(\szero)=\frac{1}{\rho_k}\left[(\sone-\szero)^{\top}\nabla \hat{f}(\szero),\; \dots, \;  (\spi-\szero)^{\top}\nabla \hat{f}(\szero)\right]^{\top}\in\rp.
\end{equation*}
Thus, the $i$th component of the vector $\Lkhat^{\top}\left(\aRandomVec-\nabla \hat{f}(\szero)\right)$ is given by 
\begin{equation*}
\ubar{\Psi}_k(\sii,\szero):=\frac{1}{\rho_k}\left(\Fsk(\sii,{\ubar{\theta}}^i)-\Fsk(\szero,{\ubar{\theta}}^0)-(\sii-\szero)^{\top}\nabla \hat{f}(\szero)\right).
\end{equation*}
We notice that 
$
\abs{{\ubar{\Psi}}_k(\sii,\szero)}
\leq {{\ubar{\Psi}}_k^{1,i}} + {\psi_k^{2,i}},
$
where 
\begin{eqnarray*}
{{\ubar{\Psi}}_k^{1,i}}&:=&\frac{1}{\rho_k}\abs{\Fsk(\sii,{\ubar{\theta}}^i)-\hat{f}(\sii)}+\frac{1}{\rho_k}\abs{\Fsk(\szero,{\ubar{\theta}}^0)-\hat{f}(\szero)}\\
\mbox{and }\quad {{{\psi}}_k^{2,i}}&:=&\frac{1}{\rho_k} \abs{\hat{f}(\sii)-\hat{f}(\szero)-(\sii-\szero)^{\top}\nabla \hat{f}(\szero)}.
\end{eqnarray*} 
It follows from Lemma~\ref{FundTheorCalc} that 
\begin{equation}\label{psiktwoIneq}
{{{\psi}}_k^{2,i}}\leq\frac{1}{2\rho_k}\Lg\Qmax^2\normii{\sii-\szero}^2 \leq\frac{1}{2}\Lg\Qmax^2\rho_k\leq \frac{1}{2}\Lg\Qmax^2\textcolor{black}{c_1}\dk.
\end{equation}
Now assume that the event $B$ occurs. Then ${\Psi_k^{1,i}}\leq \kappa_{eg}'\dk$ for all $i=0,1,\dots,p$. It follows from~\eqref{psiktwoIneq} and the inequality $\abs{{\ubar{\Psi}}_k(\sii,\szero)}^2\leq 2\left(\left({\ubar{\Psi}}_k^{1,i}\right)^2+ \left({{\psi}}_k^{2,i}\right)^2\right)$ that

\begin{equation*}
\begin{split}
\normii{\Lkhat^{\top}\left(\aRandomVec-\nabla \hat{f}(\szero)\right)}^2&= \sum_{i=1}^p{\ubar{\Psi}}_k(\sii,\szero)^2\\
&\leq 2p\kegPrimeSquare\dk^2+\frac{1}{2}p\Lg^2\Qmax^4\textcolor{black}{c_1^2}\dk^2=p\Lg^2\Qmax^4\textcolor{black}{c_1^2}\dk^2,
\end{split}
\end{equation*}
where the last equality follows from the definition of $\kappa_{eg}'$. This means that $B\subseteq E_1$.\\
$ $\\

\textbf{Part~2 ($B\subseteq E_f$)}. To show the latter inclusion, we recall that $\Mhatk(\szero)=\Fsk(\szero,{\ubar{\theta}}^0)$ and $\Mhatk(\szero)-\Mhatk(\s)=(\szero-\s)^{\top}\aRandomVec$. Then the following holds:
\begin{eqnarray*}
\abs{\hat{f}(\s)-\Mhatk(\s)}&=&\abs{\hat{f}(\s)-\Fsk(\szero,{\ubar{\theta}}^0)+\Mhatk(\szero)-\Mhatk(\s)}\\
&=&\abs{\hat{f}(\s)-\Fsk(\szero,{\ubar{\theta}}^0)+(\szero-\s)^{\top}\aRandomVec}\\
&=&\left\lvert\hat{f}(\s)-\hat{f}(\szero)-(\s-\szero)^{\top}\nabla\hat{f}(\szero)-\left(\Fsk(\szero,{\ubar{\theta}}^0)-\hat{f}(\szero)\right) +(\s-\szero)^{\top}\left(\nabla\hat{f}(\szero)-\aRandomVec\right)\right\rvert\\
&\leq& \abs{\hat{f}(\s)-\hat{f}(\szero)-(\s-\szero)^{\top}\nabla\hat{f}(\szero)}+ \abs{\Fsk(\szero,{\ubar{\theta}}^0)-\hat{f}(\szero)} + \normii{\s-\szero}\normii{\aRandomVec-\nabla\hat{f}(\szero)}.
\end{eqnarray*}
As before, the first term on the right-hand side of the last inequality is bounded by using Lemma~\ref{FundTheorCalc} as follows:
\begin{equation}\label{lemmabound2}
\abs{\hat{f}(\s)-\hat{f}(\szero)-(\s-\szero)^{\top}\nabla\hat{f}(\szero)}\leq \frac{1}{2}\Lg\Qmax^2\normii{\s-\szero}^2\leq\frac{1}{2}\Lg\Qmax^2\dk^2.
\end{equation}
To bound the last two terms, we note that the inclusions $B\subseteq B_0$ and $B\subseteq E_1\subseteq E_2$ yield  $\abs{\Fsk(\szero,{\ubar{\theta}}^0)-\hat{f}(\szero)}$ $\leq \frac{1}{2}\kappa_{eg}'\rho_k\dk^2\ $ and $\ \normii{\aRandomVec-\nabla \hat{f}(\szero)}\leq 2\sqrt{p}\kappa_{eg}'\norme{\Lkhat^{-1}}\dk$, respectively. Thus, these inequalities combined with~\eqref{lemmabound2} and the inequalities $\rho_k\leq c_1\delta_{\max}$ and $\normii{\s-\szero}\leq \dk$ lead to

\begin{eqnarray*}
\abs{\hat{f}(\s)-\Mhatk(\s)}&\leq&\frac{1}{2}\Lg\Qmax^2\dk^2+\frac{1}{2}\kappa_{eg}'\rho_k\dk^2+2\sqrt{p}\kappa_{eg}'\norme{\Lkhat^{-1}}\dk^2\\
&=&\frac{1}{2}\left(1+\frac{1}{2}\textcolor{black}{c_1^2\delta_{\max}}+2\textcolor{black}{c_1}\sqrt{p}\norme{\Lkhat^{-1}}\right)\Lg\Qmax^2\dk^2,
\end{eqnarray*}
which shows that $B\subseteq E_f$.\\
$ $\\
\textbf{Part~3$\ $} To complete the proof, we show~\eqref{BiProb} using the Chebyshev inequality as follows:
\begin{eqnarray*}
\prob{B_i^c}&=&\prob{\abs{\Fsk(\sii,{\ubar{\theta}}^i)-\Esp\left[\Fsk(\sii,{\ubar{\theta}}^i)\right]}> \frac{1}{2}\kappa_{eg}'\rho_k\min\{\dk,\dk^2\}}\\
&\leq& \frac{4 \V\left[\Fsk(\sii,{\ubar{\theta}}^i)\right]}{\kegPrimeSquare\rho_k^2\min\{\dk^2,\dk^4\}}\leq \frac{4 V_f}{\pi_k\kegPrimeSquare\rho_k^2\min\{\dk^2,\dk^4\}}\leq 1-\beta_m^{1/(p+1)},
\end{eqnarray*}
provided~\eqref{ppkChoice} holds; that is, $\pi_k$ is chosen according to
\[\pi_k\geq \frac{4 V_f}{\kegPrimeSquare\rho_k^2\min\{\dk^2,\dk^4\}\left(1-\beta_m^{1/(p+1)}\right)}= \frac{16V_f}{\textcolor{black}{c_1^2}\Lg^2\Qmax^4\rho_k^2\min(\dk^2,\dk^4)\left(1-\beta_m^{1/(p+1)}\right)}.\]
\end{proof}
The next result shows that for a particular geometry of the interpolation set $\SSk$ of Theorem~\ref{subspaceModel1}, the resulting model gradient is a forward finite-difference stochastic gradient estimator of $\Qk^{\top}\nabla f(\xk)$ and its corresponding parameters $\kef$ and $\keg$ do not depend on $k$, as mentioned above.
\begin{corollary}\label{CorQkFully}
Under all the assumptions of Theorem~\ref{subspaceModel1}, assume further that $\sii=h\bm{e_i}\in\rp, i\in \intbracket{1,p}$, where $\bm{e_i}$ is the $i$th standard basis vector of $\rp$ and $h=\min\{\hopt, \dk\}$ for some $\hopt>0$ sufficiently small. Let $\Fok(\xk):=\Fsk(0,{\ubar{\theta}}^0)$ and $\Fsk(\xk+h(\Qk)_{: i}):=\Fsk(\sii,{\ubar{\theta}}^i)$ for all $i\in \intbracket{1,p}$, where $(\Qk)_{: i}$ denotes the $i$th column of $\Qk$. Then $\Qk^{\top}\nabla f(\xk)\approx \gkhat=\Gkhat(\omega)$, where $\Gkhat$ is defined by the  forward finite-difference scheme
$(\Gkhat)_i=\frac{\Fsk(\xk+h(\Qk)_{:i})-\Fok(\xk)}{h}$, $i\in \intbracket{1,p}$.
Moreover, the corresponding model parameters~$\kef$ and~$\keg$ given by~\eqref{kefkeg} are independent of $k$.
\end{corollary}

\begin{proof}
Note that $\SSk=\accolade{\bm{0},h\bm{e_1},h\bm{e_2},\dots,h\bm{e_p}}\subseteq\rp\cap\B{0}{c_1\dk}$, $\rho_k=h$, and hence $\Lkhat=\bm{I_p}$. The proof immediately follows from~\eqref{systemThree} by replacing $\ubar{\bm{a}}$ with $\Gkhat$. The nondependence of~$\kef$ and~$\keg$ on $k$ trivially follows from the fact that $\Lkhat^{-1}=\bm{I_p}$.
\end{proof}

\section{Convergence analysis.}\label{Section4}
This section presents convergence results of Algorithm~\ref{algoStoScalTR} using ideas inspired by~\cite{audet2019stomads,billingsley1995ThirdEdition,chen2018stochastic,dzahini2020expected,dzahini2020constrained}. Section~\ref{Section4p1} presents preliminary results necessary for the proof of the main results. Section~\ref{Section4p2} proves that the sequence of random trust-region radii converges to zero almost surely. Section~\ref{Section4p3} demonstrates the existence of a subsequence of random iterates generated by the proposed method, which drives the norm of $\nabla f$ to zero almost surely. Proofs that are not presented in this section can be found in the Appendix.

\subsection{Preliminary results.}\label{Section4p1}
In the remainder of the manuscript, the following inspired by~\cite[Assumptions~4.1 and~4.3]{chen2018stochastic}, respectively, will be assumed.
\begin{assumption}\label{fGradfAssumption}
For given $\bm{\x_0}$, $\dmaxx$, and $\bm{Q_0}$, let $\mathscr{L}(\bm{\x_0};\bm{Q_0})\subset\rn$ be the set containing all iterates of Algorithm~\ref{algoStoScalTR}. Define the region $\mathscr{L}^{\star}(\bm{\x_0};\bm{Q_0})$ considered by the algorithm realizations as $\mathscr{L}^{\star}(\bm{\x_0};\bm{Q_0})=\ds{\underset{\accolade{k\, \geq\, 0,\, \xk\in\mathscr{L}(\bm{\x_0};\bm{Q_0})}}{\bigcup}}\mathcal{B}(\xk,\dmaxx;\Qk)$. Then $f(\x)\geq f_{\min}$ for all $\x\in\rn$ and for some constant $f_{\min}>-\infty$. Moreover, $f$ and its gradient are Lipschitz continuous on $\mathscr{L}^{\star}(\bm{\x_0};\bm{Q_0})$.
\end{assumption}

\begin{assumption}\label{boundedHessian}
There exists some $\kappa_h\geq 1$ such that for all $k\geq 0$, the Hessian $\Hkhat$ of all realizations of $\Mhatk$ satisfies $\norme{\Hkhat}\leq \kappa_h.$
\end{assumption}
The following subspace variant of~\cite[Lemma~4.5]{chen2018stochastic} shows that if $\dk$ is small enough compared with the size $\normii{\gkhat}$ of a $(\kef,\keg;\Qk)$-fully linear model gradient, then the trial step $\sk$ provides a decrease in $f$ proportional to $\normii{\gkhat}$.
\begin{lemma}\label{LemmaA}
Suppose that the function $\mhatk$ is a $(\kef,\keg;\Qk)$-fully linear model for~$f$ in $\B{\xk}{\dk;\Qk}$. Then the trial step $\sk$ leads to an improvement in $f$ such that
\begin{equation}\label{LemmaAimprovement}
f(\xk+\Qk\sk)-f(\xk)\leq -\frac{\kappa_{fcd}}{4}\normii{\gkhat}\dk,
\quad
\mbox{if }
\dk\leq \min\accolade{\frac{1}{\kappa_h},\frac{\kappa_{fcd}}{8\kappa_{ef}}}\normii{\gkhat}.
\end{equation}
\end{lemma}

The following subspace variant of~\cite[Lemma~4.6]{chen2018stochastic} shows that the guaranteed decrease in~$f$ provided by $\sk$ is proportional to $\normii{\QkTranspose\nabla f(\xk)}$ if $\dk$ is small enough compared with $\normii{\QkTranspose\nabla f(\xk)}$.
\begin{lemma}\label{LemmaB2}
Let Assumption~\ref{boundedHessian} hold, and suppose the model $\mhatk$ is $(\kef,\keg;\Qk)$-fully linear. Define $C_1=\frac{\kappa_{fcd}}{4} \max\accolade{\frac{\kappa_h}{\kappa_h+\kappa_{eg}}, \frac{8\kappa_{ef}}{8\kappa_{ef}+\kappa_{fcd}\kappa_{eg}}}$. Then the trial step $\sk$ leads to an improvement in $f$ such that 

\begin{equation}\label{LemmaB2improvement}
f(\xk+\Qk\sk)-f(\xk)\leq -C_1\normii{\QkTranspose\nabla f(\xk)}\dk,
\end{equation}
\begin{equation}\label{deltaLemmaB2}
\mbox{if}\quad
\dk\leq \min\accolade{\frac{1}{\kappa_h+\kappa_{eg}}, \frac{1}{\frac{8\kappa_{ef}}{\kappa_{fcd}}+\kappa_{eg}}}\normii{\QkTranspose\nabla f(\xk)}. 
\end{equation}
\end{lemma}
The next subspace variant of~\cite[Lemma~4.7]{chen2018stochastic} shows that if the model and the estimates are sufficiently accurate and  $\dk$ is small enough compared with $\normii{\gkhat}$, then the iteration is successful.
\begin{lemma}\label{LemmaC}
Let Assumption~\ref{boundedHessian} hold. Assume that $\mhatk$ is $(\kef,\keg;\Qk)$-fully linear and the estimates $\accolade{\fok,\fsk}$ are $\ef$-accurate with $\ef\leq\kappa_{ef}$. Then the $k$th iteration is successful if 
\begin{equation}\label{deltaLemmaC}
\dk\leq\min\accolade{\frac{1}{\kappa_h}, \frac{1}{\eta_2}, \frac{\kappa_{fcd}(1-\eta_1)}{8\kappa_{ef}}}\normii{\gkhat}.
\end{equation}
\end{lemma}

\begin{proof}
The proof immediately follows from that of~\cite[Lemma~4.7]{chen2018stochastic} with minor modifications and is not presented here again.
\end{proof}
The next result, which  is a subspace variant of~\cite[Lemma~4.8]{chen2018stochastic}, guarantees an amount of decrease in $f$ on true successful iterations.
\begin{lemma}\label{LemmaD}
Suppose that Assumption~\ref{boundedHessian} holds and the estimates $\accolade{\fok,\fsk}$ are $\ef$-accurate with $\ef<\frac{1}{4}\eta_1\eta_2\kappa_{fcd}\min\accolade{\frac{\eta_2}{\kappa_h}, 1}$. If the $k$th iteration is successful, then the improvement in $f$ is such that
\begin{equation}\label{LemmaDimprovement}
f(\xkun)-f(\xk)\leq -C_2\dk^2,
\end{equation}
where $C_2=\frac{1}{2}\eta_1\eta_2\kappa_{fcd}\min\accolade{\frac{\eta_2}{\kappa_h}, 1} -2\ef>0$. 
\end{lemma}

\begin{proof}
Again, the proof immediately follows from that of~\cite[Lemma~4.8]{chen2018stochastic} with minor modifications and is not presented here.
\end{proof}
Inspired by a result from the proof of~\cite[Theorem~4.11]{chen2018stochastic}, the next lemma quantifies the maximum possible amount of increase in~$f$ on false successful iterations.
\begin{lemma}\label{LemmaE}
Under Assumptions~\ref{QmaxAssumption} and~\ref{fGradfAssumption}, assume that Algorithm~\ref{algoStoScalTR} erroneously accepts a step $\sk$ leading to an increase in the objective function when $\norme{\QkTranspose\nabla f(\xk)}\geq \zeta\dk$ for some constant $\zeta>0$, and let $C_3(\zeta)=1+\frac{3L_g}{2\zeta}\Qmax^2$. Then the increase in $f$ is such that
\begin{equation}\label{LemmaEincrease}
f(\xk+\Qk\sk)-f(\xk)\leq C_3(\zeta)\norme{\QkTranspose\nabla f(\xk)}\dk.
\end{equation}
\end{lemma}

Later, to prove the key result of Theorem~\ref{zerothOrderScal}, we favor techniques derived in~\cite{paquette2018stochastic}, unlike~\cite{blanchet2016convergence,chen2018stochastic}. These techniques, also used in~\cite{audet2019stomads,dzahini2020expected,dzahini2020constrained}, make use of event indicator functions. Here, by means of Lemma~\ref{indicatorMeasurable}, we introduce a measurability result that is related to one of these indicator functions and is crucial in the present analysis. But first, the following result is required.
\begin{lemma}\label{partialMeasurable}
Any partial derivative 
$\frac{\partial f}{\partial x^i}: \left(\rn, \mathscr{B}(\rn)\right)\to \left(\R, \mathscr{B}(\R)\right),\ i\in\intbracket{1,n} $,
of a differentiable function $f$ is a Borel measurable function.
\end{lemma}

\begin{proof}
Let $(\bm{e_1},\bm{e_2},\dots,\bm{e_n})$ be the canonical basis of $\rn$. For any $i\in\intbracket{1,n} $ and $\bm{x}=(x^1,x^2,\dots,x^n)$, define, for all $j\in\N$,
$D_{ij}(\bm{x})=\frac{f(\bm{x}+h_j\bm{e_i})-f(\bm{x})}{h_j}$,
where $\accolade{h_j}_{j\in\N}$ is a sequence of positive real numbers converging to~$0$; $f$ is differentiable and hence continuous, and thus is Borel measurable~\cite[Theorem~13.2]{billingsley1995ThirdEdition}. Therefore, since for all $i\in\intbracket{1,n}$, $\underset{j\to\infty}{\lim}D_{ij}(\bm{x})=\frac{\partial f}{\partial x^i}(\bm{x}),$
then $\frac{\partial f}{\partial x^i}$ is also Borel measurable~\cite[Theorem~13.3 and Theorem~13.4-(ii)]{billingsley1995ThirdEdition}.
\end{proof}

\begin{lemma}\label{indicatorMeasurable}
Let $\zeta\geq 0$ be a constant, and consider the event $\Gamma_k:=\accolade{\Gk\geq 0}$, where $\Gk:=\normii{\QkRandomTranspose\nabla f(\Xk)}-\zeta\Dk$. Then $\Gk$, $\igk$, and $\igkc$ are $\fqalgebra$-measurable.
\end{lemma}

\begin{proof}
By construction, all the entries $\left(\QkRandom\right)_{ij},\ (i,j)\in\intbracket{1,n}\times\intbracket{1,p}$ of $\QkRandom$ are $\fqalgebra$-measurable, while $\Dk$ and the components $\Xk^i, \ i\in\intbracket{1, n}$ of $\Xk$ are $\falgebra$-measurable and hence $\fqalgebra$-measurable since $\falgebra\subseteq\fqalgebra$. Thus, for any $i\in\intbracket{1, n}$, $\frac{\partial f}{\partial x^i}(\Xk)$ is $\fqalgebra$-measurable as the composition of the Borel measurable function $\frac{\partial f}{\partial x^i}$ (from Lemma~\ref{partialMeasurable}) and the random map $\Xk:(\Omega,\fqalgebra)\to\left(\rn, \mathscr{B}(\rn)\right)$~\cite[Theorem~13.3 and Theorem~13.1-(ii)]{billingsley1995ThirdEdition}. Recall that  sums and products of measurable functions are also measurable~\cite[Theorem~13.3]{billingsley1995ThirdEdition}, 
and note that the random variable 
\begin{equation}\label{ukRandomMeasurable}
\ubar{u}_k:=\normii{\QkRandomTranspose\nabla f(\Xk)}^2=\sum_{j=1}^p\left(\sum_{i=1}^n\left(\QkRandom\right)_{ij}\frac{\partial f}{\partial x^i}(\Xk) \right)^2 
\end{equation}
is consequently $\fqalgebra$-measurable. Thus $\Gk=\sqrt{{\ubar{u}_k}}-\zeta\Dk$ is also $\fqalgebra$-measurable. The $\fqalgebra$-measurability of $\Gk$ ensures in particular that $\Gamma_k=\Gk^{-1}([0, \infty))\in\fqalgebra$ since $[0, \infty)$ is a Borel set of $\R$, whence the {\it simple} real functions $\igk$ and $\igkc$~\cite[Equation~(13.3)]{billingsley1995ThirdEdition} are also $\fqalgebra$-measurable. 
\end{proof}

\subsection{Zeroth-order convergence result.}\label{Section4p2}
To prove in Theorem~\ref{zerothOrderScal} that the sequence of random trust-region radii converges to zero almost surely, we assume the following.
\begin{assumption}\label{assumptionsOnEstimatesAndModels} 
The sequences of estimates $\accolade{\Fok, \Fsk}$ and models $\accolade{\Mhatk}$ generated by Algorithm~\ref{algoStoScalTR} are, respectively, $\beta_f$-probabilistically $\ef$-accurate for 
$\ef<\min\accolade{\kappa_{ef}, \frac{1}{4}\eta_1\eta_2\kappa_{fcd}\min\accolade{\frac{\eta_2}{\kappa_h}, 1}}$ and $\beta_m$-probabilistically $(\kef,\keg;\Qk)$-fully linear for some $\beta_f, \beta_m\in (0,1)$.
\end{assumption}

Recall Definitions~\ref{wellAlignedDeterministic} and~\ref{wellAlignedDefinition}. The result presented next shows that the sequence of trust-region radii converges to zero whether the sequence of subspace selection matrices is well aligned or not. The corresponding proof improves on the analyses presented in~\cite{blanchet2016convergence,chen2018stochastic}, for example, by reducing the number of subcases of {\bf Case~1}, discussed below, from four to three. This improvement is the outcome of a simple motivation to circumvent the derivation of a lower bound on the probability of the event that ``either the model is good and the estimates are bad, or the model is bad and the estimates are good,'' as was done in~\cite{blanchet2016convergence,chen2018stochastic}. Another remarkable difference is the fact that both subcases do not straightforwardly use the same $\sigma$-algebras for conditioning on the past, as is the case in~\cite{blanchet2016convergence,chen2018stochastic}, which results from the need to introduce a random variable $\igkhat$ related to the model gradient in {\bf Case~2}. We also note that unlike prior similar works, the remainder of the present analysis explicitly emphasizes the way all the $\sigma$-algebras $\falgebra$, $\fqalgebra$, and $\fmqalgebra$ work together for the proofs of the proposed results.

\begin{theorem}\label{zerothOrderScal}
Let all assumptions that were made in Lemmas~\ref{LemmaA}-\ref{LemmaD} hold with the same constants $C_1, C_2$, and $C_3:=C_3(\zeta)$, for some fixed $\zeta>0$ satisfying
\begin{equation}\label{zetaChoice}
\zeta\geq \kappa_{eg}+\max\accolade{\eta_2,\kappa_h, \frac{8\kappa_{ef}}{\kappa_{fcd}(1-\eta_1)}}.
\end{equation}
Let $\nu\in (0,1)$ be chosen according to 
\begin{equation}\label{nuChoiceScal}
\frac{\nu}{1-\nu} \geq \frac{8\gamma^2}{\min\accolade{C_2, \zeta C_1}}.  
\end{equation}
Assume further that Assumption~\ref{assumptionsOnEstimatesAndModels} holds with $\beta_f, \beta_m\in (0,1)$ satisfying\footnote{In~\eqref{probChoices}, the result is intentionally presented with $(1-\beta_f)+(1-\beta_m)$ instead of $2-\beta_f-\beta_m$ in order to emphasize one of its differences compared with~\cite{blanchet2016convergence,chen2018stochastic} in which a term similar to $(1-\beta_f)(1-\beta_m)$ was derived.}
\begin{equation}\label{probChoices}
\frac{\beta_f\beta_m-\frac{1}{2}}{(1-\beta_f)+(1-\beta_m)}\geq \frac{C_3}{C_1}\quad \mbox{ and }\quad \frac{\beta_f}{1-\beta_f}\geq \frac{2\nu\zeta\gamma^2C_3}{(1-\nu)(\gamma^2-1)}+2\gamma^2.
\end{equation}
Define $\varrho=\frac{1}{2}\beta_f(1-\nu)(1-\frac{1}{\gamma^2})>0$. Then, the random function $\Phik:=\nu (f(\Xk)-f_{\min})+(1-\nu)\Dk^2$ satisfies
\begin{equation}\label{PhikIncrementBound}
\E{\Phikun-\Phik|\fqalgebra}\leq -\varrho\Dk^2 \quad \mbox{for all}\ k\in\N,
\end{equation}
which implies that
\begin{equation}\label{seriesResult}
\sum_{k=0}^{\infty}\Dk^2<\infty\quad \mbox{almost surely.}
\end{equation}
\end{theorem}

\begin{proof}
The proof is inspired by those of~\cite[Theorem~3]{blanchet2016convergence} and~\cite[Theorem~4.11]{chen2018stochastic}. 
The overall goal is to prove~\eqref{PhikIncrementBound}. Indeed noticing that $\Phik>0$ and then taking expectations on both sides of the inequality in~\eqref{PhikIncrementBound}, we obtain~\eqref{seriesResult} (see, e.g.,~\cite[Theorem~3]{dzahini2020expected} for details). Let $\phi_k$ denote realizations of $\Phik$, and recall that on all successful iterations, $\xkun=\xk+\Qk\sk$ and $\dkun=\min\accolade{\gamma\dk,\dmaxx}$,  implying
\begin{equation}\label{un}
\phi_{k+1}-\phi_k\leq \nu\left(f(\xkun)-f(\xk)\right)+(1-\nu)(\gamma^2-1)\dk^2,
\end{equation}
while on unsuccessful iterations, $\xkun=\xk$ and $\dkun=\gamma^{-1}\dk$, in which case
\begin{equation}\label{bOne}
\phi_{k+1}-\phi_k\leq(1-\nu)\left(\frac{1}{\gamma^2}-1\right)\dk^2=:b_1<0.
\end{equation}

Recall the event $\Gamma_k=\accolade{\normii{\QkRandomTranspose\nabla f(\Xk)}\geq\zeta\Dk}$ of Lemma~\ref{indicatorMeasurable} with $\zeta$ satisfying ~\eqref{zetaChoice}. The proof considers two cases: $\igk=1$ and $\igkc=1$. Inspired by~\cite[proof of Theorem~4.11]{chen2018stochastic}, Case~1 aims to show that
\begin{equation}\label{huit}
\E{\igk\left(\Phikun-\Phik\right)|\fqalgebra}\leq-2\igk(1-\nu)(\gamma^2-1)\Dk^2
\leq -\frac{1}{2}\igk\beta_f(1-\nu)(1-\frac{1}{\gamma^2})\Dk^2, 
\end{equation}
where the last inequality follows from $1-\frac{1}{\gamma^2}<\gamma^2-1$ and $\beta_f\leq 1$. On the other hand, inspired by~\cite[proof of Theorem~3]{blanchet2016convergence}, Case~2 shows that
\begin{equation}\label{last}
\E{\igkc\left(\Phikun-\Phik\right)|\fqalgebra}\leq-\frac{1}{2}\igkc\beta_f(1-\nu)(1-\frac{1}{\gamma^2})\Dk^2.
\end{equation}
Thus, combining~\eqref{huit} and~\eqref{last} leads to~\eqref{PhikIncrementBound}.

$ $\\
\textbf{Case~1: $\igk=1$}\\
$ $\\
Unlike the proof of~\cite[Theorem~4.11, Case~1]{chen2018stochastic} where four subcases were considered, only three are analyzed next.

\noindent \textit{(Subcase 1i)}
Good model ($\iik=1$) and good estimates ($\ijk=1$). By~\eqref{zetaChoice}, 
\[\normii{\QkTranspose\nabla f(\xk)}\geq \max\accolade{\eta_2+\kappa_{eg}, \kappa_h+\kappa_{eg}, \frac{8\kappa_{ef}}{\kappa_{fcd}}+\kappa_{eg}}\dk,\]
implying in particular~\eqref{deltaLemmaB2} and hence a decrease in $f$ according to~\eqref{LemmaB2improvement}. Moreover, by $(\kef,\keg;\Qk)$-full linearity, 
\[\normii{\gkhat}\geq \normii{\QkTranspose\nabla f(\xk)}-\kappa_{eg}\dk\geq \max\accolade{\eta_2,\kappa_h, \frac{8\kappa_{ef}}{\kappa_{fcd}(1-\eta_1)}}\dk, \]
thus implying~\eqref{deltaLemmaC}; and since $\ef\leq \kappa_{ef}$ per assumptions of Lemma~\ref{LemmaC}, the iteration is successful. Consequently,~\eqref{un} together with~\eqref{LemmaB2improvement} yields
\begin{equation}\label{bTwo}
\begin{split}
	\phi_{k+1}-\phi_k &\leq -\nu C_1\normii{\QkTranspose\nabla f(\xk)}\dk+(1-\nu)(\gamma^2-1)\dk^2=:b_2\\
	&\leq \left(-\nu C_1\zeta+(1-\nu)(\gamma^2-1)\right)\dk^2<0,
\end{split}
\end{equation}
where the last inequality follows from~\eqref{nuChoiceScal}. Denoting by ${\ubar{b}}_2({\ubar{u}}_k,\Dk)$ the random variable with realizations $b_2<0$, with ${\ubar{u}}_k$ defined in~\eqref{ukRandomMeasurable}, we have from~\eqref{bTwo} that
\begin{equation}\label{bTwoRandom}
\igk\iik\ijk\left(\Phikun-\Phik\right)\leq \igk\iik\ijk{\ubar{b}}_2({\ubar{u}}_k,\Dk).
\end{equation}
\noindent \textit{(Subcase 1ii)}
Bad model ($\iikc=1$) and good estimates ($\ijk=1$). In this case, regardless of the iteration type (i.e., successful or unsuccessful), the change in $\phi_k$ can always be bounded by using~\eqref{un} and~\eqref{LemmaEincrease} as follows:
\begin{equation}\label{bThree}
\phi_{k+1}-\phi_k\leq\nu C_3\norme{\QkTranspose\nabla f(\xk)}\dk+(1-\nu)(\gamma^2-1)\dk^2=:b_3.
\end{equation}
Denoting by ${\ubar{b}}_3({\ubar{u}}_k,\Dk)$ the random variable with realizations $b_3>0$, from~\eqref{bThree} we have that
\begin{equation}\label{bThreeRandom2}
\igk\iikc\ijk\left(\Phikun-\Phik\right)\leq \igk\iikc\ijk{\ubar{b}}_3({\ubar{u}}_k,\Dk)\leq \igk\iikc{\ubar{b}}_3({\ubar{u}}_k,\Dk).
\end{equation}
\noindent \textit{(Subcase 1iii)}
Bad estimates ($\ijkc=1$). From \textit{Subcase 1ii}, it always holds that 
\begin{equation}\label{bThreeRandom3}
\igk\ijkc\left(\Phikun-\Phik\right)\leq \igk\ijkc{\ubar{b}}_3({\ubar{u}}_k,\Dk).
\end{equation}

With the subcases thus complete, since the random variables $\igk, {\ubar{b}}_2({\ubar{u}}_k,\Dk)$, and ${\ubar{b}}_3({\ubar{u}}_k,\Dk)$ are $\fqalgebra$-measurable thanks to the proof of Lemma~\ref{indicatorMeasurable}, combining~\eqref{bTwoRandom},~\eqref{bThreeRandom2}, and~\eqref{bThreeRandom3} and taking expectations with respect to $\fqalgebra$, we have
\begin{equation}\label{bTwoThree}
\begin{split}
\E{\igk\left(\Phikun-\Phik\right)|\fqalgebra}&\leq \igk{\ubar{b}}_2({\ubar{u}}_k,\Dk)\E{\iik\ijk|\fqalgebra}\\
&+\igk{\ubar{b}}_3({\ubar{u}}_k,\Dk)\left(\E{\iikc|\fqalgebra}+\E{\ijkc|\fqalgebra}\right)\\
&\leq \igk\left(\beta_f\beta_m{\ubar{b}}_2({\ubar{u}}_k,\Dk)+{\ubar{b}}_3({\ubar{u}}_k,\Dk)(1-\beta_m+1-\beta_f)\right)=:\igk {\ubar{b}}_{2,3}({\ubar{u}}_k,\Dk),
\end{split}
\end{equation}
where the last inequality follows from~\eqref{ikQProb} and~\eqref{prodBetafBetam} and the fact that ${\ubar{b}}_2({\ubar{u}}_k,\Dk)<0$ while ${\ubar{b}}_3({\ubar{u}}_k,\Dk)>0$.  Note that 
\begin{eqnarray*}
{\ubar{b}}_{2,3}({\ubar{u}}_k,\Dk)&=&\nu \normii{\QkRandomTranspose\nabla f(\Xk)}\Dk\left[-\beta_f\beta_m C_1+(2-\beta_f-\beta_m)C_3\right]\\
& &+(1-\nu)(\gamma^2-1)(2-\beta_f-\beta_m+\beta_f\beta_m)\Dk^2\\
&\leq& \nu \normii{\QkRandomTranspose\nabla f(\Xk)}\Dk\left[-\beta_f\beta_m C_1+(2-\beta_f-\beta_m)C_3\right] + 2(1-\nu)(\gamma^2-1)\Dk^2,
\end{eqnarray*}
where the last inequality follows from $2-\beta_f-\beta_m+\beta_f\beta_m=1+(1-\beta_f)(1-\beta_m)\leq 2$.

For $\beta_f$ and $\beta_m$ chosen according to the first condition in~\eqref{probChoices}, and for $\nu$ satisfying~\eqref{nuChoiceScal}, the following holds:
\[\beta_f\beta_m C_1 - (2-\beta_f-\beta_m)C_3\geq \frac{1}{2}C_1\geq 2\frac{2(1-\nu)(\gamma^2-1)}{\nu\zeta}.\]
Hence, since $\normii{\QkRandomTranspose\nabla f(\Xk)}\geq \zeta\Dk$, then 
\begin{eqnarray*}
{\ubar{b}}_{2,3}({\ubar{u}}_k,\Dk)&\leq& -\frac{1}{2}\left[\beta_f\beta_m C_1-(2-\beta_f-\beta_m)C_3\right]\nu \normii{\QkRandomTranspose\nabla f(\Xk)}\Dk\\
&\leq& -\frac{1}{4}C_1\nu \normii{\QkRandomTranspose\nabla f(\Xk)}\Dk\leq -2(1-\nu)(\gamma^2-1)\Dk^2,
\end{eqnarray*}
which, together with~\eqref{bTwoThree}, leads to~\eqref{huit}.

$ $\\
\textbf{Case~2: $\igkc=1$}
$ $\\ 
Consider the event $G_k:=\accolade{\normii{\Gkhat}\geq \eta_2\Dk}$, and note that since $\Mhatk:={\ubar{f}_k}+\Gkhat^{\top}\sRandom +\frac{1}{2}\sRandom^{\top}\HkhatRandom\sRandom$ is $\fmqalgebra$-measurable by construction, then in particular $\normii{\Gkhat}$ and hence $\igkhat, \igkhatc$ are also $\fmqalgebra$-measurable. Observe that if $\igkhatc=1$, then the iteration is unsuccessful since $\normii{\gkhat}<\eta_2\dk$, leading to~\eqref{bOne}. Hence, 
\begin{equation}\label{GkbarCase2}
\begin{split}
\E{\igkc\igkhatc(\Phikun-\Phik)|\fmqalgebra}&\leq -(1-\nu)\left(1-\frac{1}{\gamma^2}\right)\E{\igkc\igkhatc\Dk^2|\fmqalgebra}\\
&\leq -\frac{1}{2}\igkc\igkhatc\beta_f(1-\nu)\left(1-\frac{1}{\gamma^2}\right)\Dk^2,
\end{split}
\end{equation}
where the last inequality is due to the $\fmqalgebra$-measurability of $\igkc$ and $\Dk$, given that $\falgebra\subset\fqalgebra\subset\fmqalgebra$. 

Inspired by~\cite[Theorem~3]{blanchet2016convergence}, which improved on the proof of~\cite[Theorem~4.11, Case~2]{chen2018stochastic} in which four subcases were considered, only two are analyzed next. The overall goal is to prove that
\begin{equation}\label{GkCase2}
\E{\igkc\igkhat(\Phikun-\Phik)|\fmqalgebra}\leq -\frac{1}{2}\igkc\igkhat\beta_f(1-\nu)\left(1-\frac{1}{\gamma^2}\right)\Dk^2,
\end{equation}
which, combined with~\eqref{GkbarCase2}, leads to 
\begin{equation*}
\E{\igkc(\Phikun-\Phik)|\fmqalgebra}\leq -\frac{1}{2}\igkc\beta_f(1-\nu)\left(1-\frac{1}{\gamma^2}\right)\Dk^2
\end{equation*}
and consequently
\begin{equation*}
\begin{split}
\E{\igkc(\Phikun-\Phik)|\fqalgebra}&=\E{\E{\igkc(\Phikun-\Phik)|\fmqalgebra}|\fqalgebra}\\
&\leq \E{-\frac{1}{2}\igkc\beta_f(1-\nu)\left(1-\frac{1}{\gamma^2}\right)\Dk^2 \, |\fqalgebra}=-\frac{1}{2}\igkc\beta_f(1-\nu)\left(1-\frac{1}{\gamma^2}\right)\Dk^2,
\end{split}
\end{equation*}
 showing~\eqref{last}, where the last equality follows from the $\fqalgebra$-measurability of $\igkc$ and $\Dk$.  What remains to be shown is~\eqref{GkCase2} as achieved next, assuming that $\igkhat=1$.

\noindent \textit{(Subcase 2i)}
Good estimates ($\ijk=1$) and $\igkhat=1$. While $f$ is decreased on successful iterations because  of good estimates thanks to Lemma~\ref{LemmaD} and in this case~\eqref{LemmaDimprovement} holds,  $\dk$ is reduced on unsuccessful iterations, leading to~\eqref{bOne}. However, combining~\eqref{LemmaDimprovement} and~\eqref{un} yields
$\phi_{k+1}-\phi_k\leq \left[-\nu C_2+(1-\nu)(\gamma^2-1)\right]\dk^2\leq b_1,$
where the last inequality is due to~\eqref{nuChoiceScal}, which shows that~\eqref{bOne} holds in any case. Consequently,
\begin{equation}
\label{GammakBarGkJk}
\igkc\igkhat\ijk(\Phikun-\Phik) \leq -\igkc\igkhat\ijk(1-\nu)\left(1-\frac{1}{\gamma^2}\right)\Dk^2.
\end{equation}

\noindent \textit{(Subcase 2ii)}
Bad estimates ($\ijkc=1$) and $\igkhat=1$. In this case, a successful step can lead to an increase in $f$
according to~\eqref{LemmaEincrease}, which combined with~\eqref{un} yield~\eqref{bThree} and hence
\begin{equation}\label{GammakBarGkJkBar}
\igkc\igkhat\ijkc(\Phikun-\Phik) \leq -\igkc\igkhat\ijkc\left(\nu C_3\zeta+(1-\nu)(\gamma^2-1)\right)\Dk^2.
\end{equation}

With the subcases complete, recall that $\E{\ijk|\fmqalgebra}\geq\beta_f$. Then  combining~\eqref{GammakBarGkJk} and~\eqref{GammakBarGkJkBar} and taking expectations with respect to~$\fmqalgebra$ lead to
\begin{eqnarray*}
\E{\igkc\igkhat(\Phikun-\Phik)|\fmqalgebra}&\leq& \igkc\igkhat\left(-\beta_f(1-\nu)(1-1/\gamma^2)\right.\\ 
& &\left. +\ (1-\beta_f)\left(\nu C_3\zeta+(1-\nu)(\gamma^2-1)\right) \right)\Dk^2\\
&\leq& -\frac{1}{2}\igkc\igkhat\beta_f(1-\nu)\left(1-\frac{1}{\gamma^2}\right)\Dk^2,
\end{eqnarray*}
thus proving~\eqref{GkCase2}, where the last inequality follows from the second condition in~\eqref{probChoices}, and the proof is complete.
\end{proof}

\subsection{Liminf-type convergence.}\label{Section4p3}
The next result demonstrates the existence of a subsequence of random iterates generated by Algorithm~\ref{algoStoScalTR}, which drives $\nabla f$ to zero almost surely. While the corresponding proof is inspired by that of~\cite[Theorem~4.16]{chen2018stochastic}, we point out the additional difficulty introduced in the present work by the random matrices $\QkRandom$. Unlike~\cite[Section~4]{chen2018stochastic} where ``auxiliary lemmas'' similar to those of Section~\ref{Section4p1} are directly related to $\normii{\nabla f(\xk)}$ and the size $\normii{\bm{g_k}}$ of a full space model gradient, thus easing the proof of the liminf-type result, this is not the case here. Instead, recovering $\normii{\nabla f(\xk)}$ through $\normii{\QkTranspose \nabla f(\xk)}$ by means of the well-alignment assumption becomes crucial.
\begin{theorem}\label{liminfTheorem}
Let all the assumptions made in Theorem~\ref{zerothOrderScal} hold. Assume further that Assumption~\ref{wellAlignedAssumption} holds with $0\leq\beta_Q<1-\frac{1}{2\beta_f\beta_m}$. Then, the sequence $\accolade{\Xk}_{k\in\N}$ of random iterates generated by Algorithm~\ref{algoStoScalTR} satisfies
\[\underset{k\to\infty}{\liminf}\normii{\nabla f(\Xk)}=0\quad\mbox{ almost surely.}\]
\end{theorem}

\begin{proof}
The result is proved by contradiction conditioned on the almost sure event $E_0:=\accolade{\Dk\to 0}$ (thanks to Theorem~\ref{zerothOrderScal}), inspired by the proof of~\cite[Theorem~4.16]{chen2018stochastic} and also using ideas from~\cite{audet2019stomads,dzahini2020constrained}. Assume that with nonzero probability there exists a random variable ${\ubar{\varepsilon}}'>0$ such that 
$\normii{\nabla f(\Xk)}\geq {\ubar{\varepsilon}}' \mbox{ for all } k\in\N.$

Let $\accolade{\xk}, \accolade{\dk}$, and ${{\varepsilon}}'$ be realizations of $\accolade{\Xk}, \accolade{\Dk}$, and ${\ubar{\varepsilon}}'$, respectively, for which $\normii{\nabla f(\xk)}\geq {{\varepsilon}}'$ for all $k$. Since $\dk\to 0$, there exists $k_0\in\N$ such that
\begin{equation}\label{bChoice}
\dk<b:=\alpha_Q\cdot\min\accolade{\frac{{{\varepsilon}}'}{2(\kappa_h+\kappa_{eg})}, \frac{{{\varepsilon}}'}{2(\eta_2+\kappa_{eg})}, \frac{{{\varepsilon}}'}{\frac{16\kappa_{ef}}{\kappa_{fcd}(1-\eta_1)}+2\kappa_{eg}}, \frac{\dmaxx}{\gamma}}
\end{equation}
for all $k\geq k_0$.
Let $\ubar{r}_k$ be the random variable with realizations $r_k=\log_{\gamma}\left(\frac{\dk}{b}\right)$. Then $r_k<0$ for all $k\geq k_0$. The main idea is to show that such realizations occur only with probability zero, leading to a contradiction. 

To prove that $\accolade{\ubar{r}_k}$ is a submartingale, recall the events $I^Q_k$ and $J^Q_k$ satisfying~\eqref{otherProbIneqs} and $\mathcal{A}_k$ satisfying~\eqref{Akineq}.
Consider some iteration $k\geq k_0$ for which $\mathcal{A}_k$, $I^Q_k$, and $J^Q_k$ all occur, which happens with probability at least $\beta_f\beta_m(1-\beta_Q)$ conditioned on $\falgebra$, as was shown in~\eqref{prodBetaIneq}. It follows from~\eqref{bChoice} that 
\begin{equation}\label{ineq4}
\normii{\nabla f(\xk)}\geq \frac{1}{\alpha_Q}\max\accolade{\kappa_h+\kappa_{eg}, \eta_2+\kappa_{eg}, \frac{8\kappa_{ef}}{\kappa_{fcd}(1-\eta_1)}+\kappa_{eg}}\dk,
\end{equation}
which, together with the $(\kef,\keg;\Qk)$-full linearity and $\alpha_Q$-well alignment, yield
\begin{equation}\label{ineq5}
\normii{\gkhat}\geq\normii{\QkTranspose\nabla f(\xk)}-\kappa_{eg}\dk\geq \alpha_Q \normii{\nabla f(\xk)}-\kappa_{eg}\dk
\geq \max\accolade{\kappa_h, \eta_2, \frac{8\kappa_{ef}}{\kappa_{fcd}(1-\eta_1)}}\dk.
\end{equation}
Thus, the $k$th iteration is successful according to Lemma~\ref{LemmaC}. Hence, $\dkun=\gamma\dk$, and hence $r_{k+1}=r_k+1$. For all other outcomes of $\mathcal{A}_k$, $I^Q_k$, and $J^Q_k$, which occur with a total probability of at most $1-\beta_f\beta_m(1-\beta_Q)$, it always holds that $\dkun\geq\gamma^{-1}\dk$, which implies that $r_{k+1}\geq r_k-1$. Hence, 
\begin{eqnarray*}
\E{\mathds{1}_{\mathcal{A}_k\cap I^Q_k\cap J^Q_k} \left(\ubar{r}_{k+1}-\ubar{r}_k\right)|\falgebra}&\geq& \beta_f\beta_m(1-\beta_Q)\quad \quad\mbox{and}\\
\E{\mathds{1}_{\overline{\mathcal{A}_k\cap I^Q_k\cap J^Q_k}} \left(\ubar{r}_{k+1}-\ubar{r}_k\right)|\falgebra}&\geq& \beta_f\beta_m(1-\beta_Q)-1, 
\end{eqnarray*}
which shows that $\accolade{\ubar{r}_k}$ is a submartingale if $0\leq\beta_Q<1-\frac{1}{2\beta_f\beta_m}$.

Using the random walk defined in~\eqref{randomWalk}, one  can easily  show (following, e.g.,~\cite[Theorem~4]{audet2019stomads}, \cite[Theorem~4.16]{chen2018stochastic} and~\cite[Theorem~3.6 and Lemmas~4.7 and~4.12]{dzahini2020constrained}) that $\accolade{\ubar{w}_k}$ is also a submartingale with bounded increments and, as such, cannot converge to a finite value, and therefore $\prob{\underset{k\to\infty}{\limsup}\ \ubar{w}_k=\infty}=1$ (thanks to~\cite[Theorem~4.4]{chen2018stochastic}). However, since $r_k-r_{k_0}\geq w_k-w_{k_0}$ by construction,  the sequence of realizations~$r_k$ such that $r_k<0\ \forall k\ \geq k_0$ occurs with probability zero, leading to a contradiction, which achieves the proof.
\end{proof}

\section{Expected complexity analysis.}\label{Section5}
Section~\ref{Section5p1} introduces relevant assumptions, definitions, and theorems derived in the analysis of a general renewal-reward stochastic process and its associated stopping time introduced in~\cite{blanchet2016convergence} for the expected complexity analysis of a stochastic trust-region method. That renewal-reward process was also used in~\cite{dzahini2020expected,paquette2018stochastic} for the  analyses of stochastic direct-search and line-search methods. We show  in Section~\ref{ExpectedSubsection}  how the aforementioned assumptions are satisfied for Algorithm~\ref{algoStoScalTR} and we bound the expected number of iterations required to achieve  $\norme{\nabla f(\Xk)}\leq \varepsilon$.

\subsection{A renewal-reward martingale process.}\label{Section5p1}

A formal definition from~\cite{blanchet2016convergence} of the stopping time related to a discrete time stochastic process is as follows.
\begin{definition}
A random variable~$\tauEpsilon$ is called a stopping time with respect to a given discrete time stochastic process $\accolade{\X_k}_{k\in\N}$ if the event $\accolade{\tauEpsilon=k}$ belongs to the $\sigma$-algebra $\sigma\left(\X_1,\dots,\X_k \right)$ generated by $\X_1,\dots,\X_k$, for each $k\in\N$.
\end{definition}

Consider a stochastic process $\accolade{\left(\Phik,\Dk\right)}_{k\in\N}$, where $\Phik, \Dk\in [0,\infty)$, and introduce on the same probability space as $\accolade{\left(\Phik,\Dk\right)}_{k\in\N}$ a biased random walk process $\accolade{\Wtildek}_{k\in\N}$ obeying the following dynamics:
\begin{equation}\label{dynamicsWk}
\prob{\Wtildekun=1|\F_k}=q\qquad \mbox{ and }\qquad \prob{\Wtildekun=-1|\F_k}=1-q,
\end{equation}
where $q\in (1/2, 1)$ and $\F_k=\sigma\left( \left(\PhiZero,\Dzero,\WtildeZero\right),\dots,\left(\Phik,\Dk,\Wtildek\right) \right)$, with ${\WtildeZero=1}$.

Let $\accolade{\tauEpsilon}_{\epsilon>0}$ be a family of stopping times with respect to $\accoladekinN{\F_k}$, parameterized by $\epsilon$. A bound on $\E{\tauEpsilon}$ is derived in~\cite{blanchet2016convergence,dzahini2020expected,paquette2018stochastic} under the following assumption.

\begin{assumption}\label{assumptionRenewal}
The following hold for the stochastic process $\accoladekinN{\left(\Phik,\Dk,\Wtildek\right)}$.
\begin{itemize}
\item[(i)] There exist $\lambda\in (0, \infty)$ and $\dmaxx=\delta_0 e^{\lambda j_{\max}}$ for some $j_{\max}\in\Z$ such that $\Dk\leq \dmaxx$ for all $k$.
\item[(ii)] There exists $\depsilon=\delta_0 e^{\lambda j_{\epsilon}}$, for some $j_{\epsilon}\in\Z$ with $j_{\epsilon}\leq 0$ such that for all $k$
\begin{equation}\label{dynamicsRenewal}
\itksup\Dkun\geq \itksup\min\accolade{\Dk e^{\lambda\Wtildekun}, \depsilon}, 
\end{equation}
where $\Wtildekun$ satisfies~\eqref{dynamicsWk}. 
\item[(iii)] There exists a nondecreasing function $h:[0,\infty)\to (0,\infty)$ and a constant $\varrho>0$ such that for all $k$
\[\E{\Phikun-\Phik|\F_k}\itksup\leq-\varrho h\left(\Dk\right)\itksup.\]
\end{itemize}
\end{assumption}

Noticing that the event $\accolade{\Dk\geq\depsilon}$ occurs sufficiently frequently on average, often $\E{\Phikun-\Phik}$ can be bounded by some negative fixed constant \cite{blanchet2016convergence}, thus allowing a bound on the expected stopping time $\E{\tauEpsilon}$, as stated next.

\begin{theorem}\label{theorComplexity1}
Under Assumption~\ref{assumptionRenewal}, 
$\E{\tauEpsilon}\leq \frac{q}{2q-1}\cdot\frac{\PhiZero}{\varrho h\left(\depsilon\right)}+1.$
\end{theorem}

\subsection{Expected complexity result.}\label{ExpectedSubsection}

Consider the process $\accolade{\left(\Phik,\Dk,\Wtildek\right)}_{k\in\N}$, where~$\Phik$ is defined in Theorem~\ref{zerothOrderScal}, $\Dk$ is the trust-region radius, and $\Wtildek$ is  defined by $\Wtildek=2\left(\iak\iik\ijk-\frac{1}{2} \right)$.  Given $\epsilon\in (0,1)$, let $\tauEpsilon$ be the number of iterations required by Algorithm~\ref{algoStoScalTR} to first drive the norm of the gradient of~$f$ below $\epsilon$:  
\[\tauEpsilon:=\inf \accolade{k\in\N: \normii{\nabla f(\Xk)}\leq\epsilon}.\]
Then $\tauEpsilon$ is a stopping time for the stochastic process generated by Algorithm~\ref{algoStoScalTR}, and hence for $\accolade{\left(\Phik,\Dk,\Wtildek\right)}_{k\in\N}$ \cite{blanchet2016convergence,dzahini2020expected,paquette2018stochastic}. More precisely, the occurrence of $\accolade{\tauEpsilon=k}$ can be determined by observing $\left(\PhiZero,\Dzero,\WtildeZero\right)$, $\dots$ $(\Phikmun,\Dkmun,\Wtildekmun)$, which means that $\tauEpsilon$ is a stopping time with respect to the filtration $\accoladekinN{\falgebra}$. To bound $\E{\tauEpsilon}$ using Theorem~\ref{theorComplexity1}, we show next that Assumption~\ref{assumptionRenewal} holds for $\accolade{\left(\Phik,\Dk,\Wtildek\right)}_{k\in\N}$.

From~\eqref{PhikIncrementBound}, since $\falgebra\subset\fqalgebra$, we observe that\footnote{Unlike the analysis in~\cite{blanchet2016convergence}, the first equality of~\eqref{PhikIncrSecond} is essential and due to the fact that in the present framework $\tauEpsilon$ is a stopping time with respect to $\sigma$-algebras $\falgebra$ that are smaller than $\fqalgebra$ with respect to which $\accoladekinN{\Phik}$ was proved in Theorem~\ref{zerothOrderScal} Eq.~\eqref{PhikIncrementBound} to be a supermartingale.} for all $k\in\N$:
\begin{equation}\label{PhikIncrSecond}
\E{\Phikun-\Phik|\falgebra}=\E{\E{\Phikun-\Phik|\fqalgebra}|\falgebra}\leq\E{-\varrho\Dk^2|\falgebra}=-\varrho\Dk^2.
\end{equation}
Note that~\eqref{PhikIncrSecond} holds for any realizations of the random variables $\E{\Phikun-\Phik|\falgebra}$ and $\Dk^2$ and hence on the event $\accolade{\tauEpsilon >k}$ in particular, which shows that Assumption~\ref{assumptionRenewal}-$(iii)$ holds with the constant $\varrho=\frac{1}{2}\beta_f(1-\nu)(1-\frac{1}{\gamma^2})$ of Theorem~\ref{zerothOrderScal} and $h(t)=t^2$. Assumption~\ref{assumptionRenewal}-$(i)$ automatically follows from the initialization strategy in Algorithm~\ref{algoStoScalTR} with $\lambda=\ln(\gamma)$. Before showing by means of Lemma~\ref{smallLemma}, inspired by~\cite[Lemma~7]{blanchet2016convergence}, that~\eqref{dynamicsRenewal} is satisfied, we first define the constant 
\begin{equation}\label{deltaEpsilonDefinition}
\depsilon:=\frac{\epsilon}{\xi}, \quad\mbox{with } \xi\geq  \frac{1}{\alpha_Q}\max\accolade{\kappa_h+\kappa_{eg}, \eta_2+\kappa_{eg}, \frac{8\kappa_{ef}}{\kappa_{fcd}(1-\eta_1)}+\kappa_{eg}}=:\hat{\xi}, 
\end{equation}
inspired by the proof of Theorem~\ref{liminfTheorem}, especially~\eqref{ineq4}. Then, following ~\cite{blanchet2016convergence} exactly, one can assume without loss of generality that $\depsilon=\gamma^i\delta_0$ for some integer $i=:j_{\epsilon}\leq 0$, whence $\Dk=\gamma^{{\ubar{i}}_k}\depsilon$ for any $k$ and some integer~${\ubar{i}}_k$.

\begin{lemma}\label{smallLemma}
Let all the assumptions made in Theorem~\ref{liminfTheorem} hold. Then~\eqref{dynamicsRenewal} is satisfied for\\ $\Wtildek=2\left(\iak\iik\ijk-\frac{1}{2} \right)$, with $\lambda=\ln(\gamma)$ and some fixed $q\in (\tilde{\beta}, 1)$ with $\tilde{\beta}=\beta_f\beta_m(1-\beta_Q)$.
\end{lemma}

\begin{proof}
The result is proved by suitably adapting the proof of~\cite[Lemma~7]{blanchet2016convergence} as was done in~\cite{dzahini2020expected}. First, we notice that~\eqref{dynamicsRenewal} trivially holds when $\itksup=0$. Next, we show that when $\itksup=1$, then 
\[\Dkun\geq \min\accolade{\depsilon, \min\accolade{\dmaxx,\gamma\Dk}\iak\iik\ijk+\gamma^{-1}\Dk\left(1-\iak\iik\ijk\right)}. \]
Recall that $\dk=\gamma^{i_k}\depsilon$ for some integer $i_k$. Therefore, if $\dk>\depsilon$, then $\dk\geq \gamma\depsilon$, which implies that $\dkun\geq\gamma^{-1}\dk\geq\depsilon$. Now assume that $\dk\leq \depsilon$. Since $\tauEpsilon>k$, we have $\normii{\nabla f(\xk)}>\epsilon=\xi\depsilon\geq\xi\dk$. Thus, it follows from the definition of~$\xi$ that~\eqref{ineq4} holds. If $\iak=1$, $\iik=1$, and $\ijk=1$, then~\eqref{ineq5} holds, and consequently the $k$th iteration is successful, as was  explained in the proof of Theorem~\ref{liminfTheorem}. Hence, $\xkun=\xk+\Qk\sk$ and $\dkun=\min\accolade{\dmaxx,\gamma\dk}$. If $\iak\iik\ijk=0$, then it always holds that $\dkun\geq \gamma^{-1}\dk$. The proof is completed by  observing that $\prob{\iak\iik\ijk=1}=q$ for some fixed $q\geq \tilde{\beta}$.
\end{proof}
The main complexity result is provided by the next  theorem.

\begin{theorem}
Let all the assumptions made in Theorem~\ref{liminfTheorem} hold. Then 
\[\E{\tauEpsilon}\leq \frac{\tilde{\beta}}{2\tilde{\beta}-1}\cdot\frac{\PhiZero\tilde{\xi}^2}{\varrho \epsilon^2}+1\]
for some $\tilde{\xi}\geq\hat{\xi}$, where $\varrho$ is the constant of Theorem~\ref{zerothOrderScal}, $\tilde{\beta}$ is the same constant of Lemma~\ref{smallLemma}, and $\hat{\xi}$ is defined in~\eqref{deltaEpsilonDefinition}.
\end{theorem}

\begin{proof}
Recall the choice of $q$ in Lemma~\ref{smallLemma}, and pick some $\tilde{\xi}\geq\hat{\xi}$ such that $\depsilon=\frac{\epsilon}{\tilde{\xi}}=\gamma^{i}\delta_0$ for some integer $i\leq 0$ (without loss of generality), as discussed above. The proof follows by employing Theorem~\ref{theorComplexity1} with $h(t)=t^2$, which yields
\[\E{\tauEpsilon}\leq \frac{q}{2q-1}\cdot\frac{\PhiZero\tilde{\xi}^2}{\varrho \epsilon^2}+1\leq \frac{\tilde{\beta}}{2\tilde{\beta}-1}\cdot\frac{\PhiZero\tilde{\xi}^2}{\varrho \epsilon^2}+1.\]
\end{proof}

\input{numerical}

\section*{Conclusion.} 
This work introduces STARS, the first DFO algorithm developed for stochastic objective functions that achieves scalability using random models constructed in low-dimensional random subspaces. The analysis of STARS extends an existing framework of model-based stochastic DFO (where randomness comes only from the stochasticity of the objective function), to settings where additional randomness stems from the mechanics of the algorithm. Making use of a supermartingale-based framework, we prove that the expected complexity of STARS, which uses subspace models, is similar to that of stochastic DFO algorithms in a smooth nonconvex setting. 
Numerical experiments demonstrate the performance of STARS on large-scale problems using linear interpolation models in subspaces of various dimensions.

We note that for ease of exposition we have focused in
\eqref{probl1} on $f_{\ubar{\theta}}$ being an unbiased estimator of $f$, but this can be relaxed in all ways that the STORM framework can address. Similarly, for concreteness, the stated algorithm and numerical results have focused respectively on linear and quadratic models, but the analysis readily applies to more general probabilistically fully linear models in subspaces.
Apart from the use of particular random interpolation models, the present work also employs a classical Monte Carlo sampling strategy for computing estimates. Future works can improve this sampling strategy while also using more general random models.

\bibliographystyle{siamplain}
\bibliography{references}

\section*{Appendix.}
This appendix presents the proofs of a series of results in the main body of the manuscript. 
\subsubsection*{Proof of Theorem~\ref{sHashingTheorem}.}
\begin{proof}
Since a detailed proof is provided in~\cite{KaNel2014SparseLidenstrauss}, only its main idea is presented here. We note as mentioned in~\cite[Section~1.1]{KaNel2014SparseLidenstrauss} that one can  assume without any loss of generality that $\normii{\vi}=1$, in which case the result follows by showing that $\normii{\SMatrix\vi}^2\in\left[(1-\varepsilon)^2, (1+\varepsilon)^2\right]$, which is implied by $\abs{\normii{\SMatrix\vi}^2-1}\leq 2\varepsilon-\varepsilon^2$. Thus, it suffices to show that for any unit norm $\vi$, 
$\prob{\abs{\normii{\SMatrix\vi}^2-1}>2\varepsilon-\varepsilon^2}<\beta,$ 
which is proved in~\cite[Theorem~4.3]{KaNel2014SparseLidenstrauss} for the above choices of $\ell, r$, and $p$.
\end{proof}
\subsubsection*{Proof of Lemma~\ref{FundTheorCalc}.}
\begin{proof}
First, we note that 
\begin{eqnarray*}
\left\lvert\hat{f}(\s+\di)-\hat{f}(\s)-\di^{\top}\nabla \hat{f}(\s)\right\rvert&=&\left\lvert f(\xk+\Qk(\s+\di)) - f(\xk+\Qk\s)- \di^{\top}\Qk^{\top}\nabla f(\xk+\Qk\s)\right\rvert\\
&=&\abs{f(\bm{{y}}+\dik)-f(\bm{{y}})-\dik^{\top}\nabla f(\bm{{y}})},
\end{eqnarray*}
where $\bm{{y}}=\xk+\Qk\s\in\rn$ and $\dik=\Qk\di\in\rn$. It follows from the Fundamental Theorem of Calculus in $\rn$ (see, e.g.,~\cite[Lemma~9.3]{AuHa2017}) that for all $\bm{{y}}, \dik\in\rn$,
$f(\bm{{y}}+\dik)-f(\bm{{y}})=\int_{0}^{1}\dik^{\top}\nabla f(\bm{{y}}+\tau \dik)d\tau.$ Thus, 
\begin{eqnarray*}
\abs{\hat{f}(\s+\di)-\hat{f}(\s)-\di^{\top}\nabla \hat{f}(\s)}&=& \abs{\int_{0}^{1}\dik^{\top}\left[\nabla f(\bm{{y}}+\tau \dik) - \nabla f(\bm{{y}})\right]d\tau}\leq\\
\int_{0}^{1}\normii{\dik}\normii{\nabla f(\bm{{y}}+\tau \dik) - \nabla f(\bm{{y}})}d\tau
&\leq& \Lg\normii{\dik}^2\int_{0}^{1}\tau d\tau\leq \frac{1}{2}\Lg\Qmax^2\normii{\di}^2.
\end{eqnarray*}
\end{proof}
\subsubsection*{Proof of Lemma~\ref{LemmaA}.}
\begin{proof}
The proof is similar to that of~\cite[Lemma~4.5]{chen2018stochastic} and is not detailed here. It follows from~\eqref{CauchyDecrease} together with the inequalities $\kappa_h\geq \max\accolade{\norme{\Hkhat}, 1}$ and $\normii{\gkhat}\geq\kappa_h\dk$ that  $\mhatk(\bm{0})-\mhatk(\sk)\geq\frac{\kappa_{fcd}}{2}\normii{\gkhat}\dk.$ Since the model is $(\kef,\keg;\Qk)$-fully linear, then 
$f(\xk+\Qk\sk)-f(\xk)\leq 2\kappa_{ef}\dk^2+\mhatk(\sk)-\mhatk(\bm{0})\leq -\frac{\kappa_{fcd}}{4}\normii{\gkhat}\dk,$
where the last inequality follows from the fact that $\dk\leq \frac{\kappa_{fcd}}{8\kappa_{ef}}\normii{\gkhat}$.
\end{proof}
\subsubsection*{Proof of Lemma~\ref{LemmaB2}.}
\begin{proof}
The proof is inspired by (and similar to) that of~\cite[Lemma~4.6]{chen2018stochastic}. By $(\kef,\keg;\Qk)$-full linearity and~\eqref{deltaLemmaB2}, it holds
\begin{equation}\label{ineq1}
\normii{\gkhat}\geq \normii{\QkTranspose\nabla f(\xk)}-\kappa_{eg}\dk\geq \max\accolade{\kappa_h, \frac{8\kappa_{ef}}{\kappa_{fcd}}}\dk, 
\end{equation}
which shows that condition~\eqref{LemmaAimprovement} is satisfied. Consequently $f$ decreases as in~\eqref{LemmaAimprovement}. The first inequality in~\eqref{ineq1}, together with~\eqref{deltaLemmaB2}, yields
$\normii{\gkhat}\geq \frac{4C_1}{\kappa_{fcd}}\normii{\QkTranspose\nabla f(\xk)},$
which, combined with~\eqref{LemmaAimprovement}, implies~\eqref{LemmaB2improvement}, and the proof is complete.
\end{proof}
\subsubsection*{Proof of Lemma~\ref{LemmaE}.}
\begin{proof}
According to the Fundamental Theorem of Calculus in $\rn$ (see, e.g.,~\cite[Lemma~9.3]{AuHa2017}),\\
$\abs{f(\xk)-f(\xk+\Qk\sk)-\nabla f(\xk+\Qk\sk)^{\top}(-\Qk\sk)}\leq\frac{1}{2}L_g\Qmax^2\dk^2,$  with the inequality also following from $\normii{\Qk\sk}\leq \Qmax\dk$, which implies that
\begin{equation}\label{ineq2}
f(\xk+\Qk\sk)-f(\xk)\leq \sk^{\top}\left[\QkTranspose \nabla f(\xk+\Qk\sk)\right] + \frac{1}{2}L_g\Qmax^2\dk^2. 
\end{equation}
Per Lipschitz continuity of $\nabla f$, 
$\normii{\QkTranspose \nabla f(\xk+\Qk\sk)-\QkTranspose \nabla f(\xk)} \leq L_g\Qmax^2\dk,$
which implies that
\begin{equation}\label{ineq3}
\normii{\QkTranspose \nabla f(\xk+\Qk\sk)}\leq L_g\Qmax^2\dk+\normii{\QkTranspose \nabla f(\xk)}.
\end{equation}
Thus, using~\eqref{ineq2} and~\eqref{ineq3}, together with the inequality $\dk\leq \zeta^{-1}\normii{\QkTranspose \nabla f(\xk)}$, yields
\begin{eqnarray*}
f(\xk+\Qk\sk)&-&f(\xk)\leq \normii{\QkTranspose \nabla f(\xk+\Qk\sk)}\dk+\frac{L_g}{2\zeta}\Qmax^2\normii{\QkTranspose \nabla f(\xk)}\dk\\
&\leq& \left(\frac{L_g}{\zeta}\Qmax^2\normii{\QkTranspose \nabla f(\xk)} +\normii{\QkTranspose \nabla f(\xk)}\right)\dk\\
&+&\frac{L_g}{2\zeta}\Qmax^2\normii{\QkTranspose \nabla f(\xk)}\dk= \left(1+\frac{3L_g}{2\zeta}\Qmax^2\right)\normii{\QkTranspose \nabla f(\xk)}\dk,
\end{eqnarray*}
which completes the proof.
\end{proof}
\begin{table}[tb]
\centering
	
\caption{The 40 problems considered. \label{tab:probs}}
\begin{adjustbox}{width=12.5cm,center}
\begin{tabular}{@{}lllll|lllll@{}}
No. & Problem  & $n$   & $m$   & $\hopt$     & No. & Problem  & $n$   & $m$   & $\hopt$     \\ \hline
1  & ARGLALE  & $100$ & $200$ & $1e$-$03$ & 21 & EIGENC   & $110$ & $110$ & $1e$-$04$ \\
2  & ARGLBLE  & $100$ & $200$ & $5e$-$04$ & 22 & EXTROSNB & $100$ & $100$ & $1e$-$04$ \\
3  & ARGLCLE  & $100$ & $200$ & $5e$-$04$ & 23 & FREUROTH & $100$ & $198$ & $2e$-$04$ \\
4  & ARTIF    & $100$ & $100$ & $4e$-$05$ & 24 & INTEGREQ & $100$ & $100$ & $1e$-$05$ \\
5  & ARWHDNE  & $100$ & $198$ & $1e$-$04$ & 25 & MANCINO  & $100$ & $100$ & $2e$-$03$ \\
6  & BDQRTIC  & $100$ & $192$ & $9e$-$05$ & 26 & MOREBV   & $100$ & $100$ & $1e$-$07$ \\
7  & BDVALUES & $100$ & $100$ & $1e$-$02$ & 27 & MSQRTA   & $100$ & $100$ & $2e$-$04$ \\
8  & BRATU2D  & $100$ & $100$ & $4e$-$06$ & 28 & MSQRTB   & $100$ & $100$ & $2e$-$04$ \\
9  & BRATU2DT & $100$ & $100$ & $5e$-$06$ & 29 & OSCIGRNE & $100$ & $100$ & $9e$-$05$ \\
10 & BRATU3D  & $125$ & $125$ & $1e$-$05$ & 30 & Penalty2 & $100$ & $200$ & $5e$-$05$ \\
11 & BROWNALE & $100$ & $100$ & $4e$-$04$ & 31 & PENLT1NE & $100$ & $101$ & $8e$-$03$ \\
12 & BROYDN3D & $100$ & $100$ & $4e$-$05$ & 32 & POWELLSE & $100$ & $100$ & $3e$-$04$ \\
13 & BROYDNBD & $100$ & $100$ & $4e$-$05$ & 33 & POWELLSG & $100$ & $100$ & $2e$-$04$ \\
14 & CBRATU2D & $98$  & $98$  & $4e$-$06$ & 34 & ROSENBR  & $100$ & $198$ & $1e$-$04$ \\
15 & CHANDHEQ & $100$ & $100$ & $6e$-$05$ & 35 & SPMSQRT  & $100$ & $164$ & $3e$-$04$ \\
16 & CHEBYQAD & $100$ & $100$ & $4e$-$06$ & 36 & SROSENBR & $100$ & $100$ & $4e$-$05$ \\
17 & ConnBand & $100$ & $100$ & $5e$-$05$ & 37 & VARDIMNE & $100$ & $102$ & $1e$-$04$ \\
18 & CUBE     & $100$ & $100$ & $4e$-$05$ & 38 & VarTrig  & $100$ & $100$ & $3e$-$03$ \\
19 & EIGENA   & $110$ & $110$ & $2e$-$04$ & 39 & YATP1SQ  & $99$  & $99$  & $9e$-$04$ \\
20 & EIGENB   & $110$ & $110$ & $6e$-$05$ & 40 & YATP2SQ  & $99$  & $99$  & $9e$-$04$
\end{tabular}
\end{adjustbox}
\end{table}

\clearpage

\framebox{\parbox{.90\linewidth}{\scriptsize The submitted manuscript has been created by
        UChicago Argonne, LLC, Operator of Argonne National Laboratory (``Argonne'').
        Argonne, a U.S.\ Department of Energy Office of Science laboratory, is operated
        under Contract No.\ DE-AC02-06CH11357.  The U.S.\ Government retains for itself,
        and others acting on its behalf, a paid-up nonexclusive, irrevocable worldwide
        license in said article to reproduce, prepare derivative works, distribute
        copies to the public, and perform publicly and display publicly, by or on
        behalf of the Government.  The Department of Energy will provide public access
        to these results of federally sponsored research in accordance with the DOE
        Public Access Plan \url{http://energy.gov/downloads/doe-public-access-plan}.}}

%% file: numerical.tex
\section{Numerical results.}\label{Section6}
We now illustrate the performance of STARS for various random subspace and full space (i.e., STORM-like) forms. We consider stochastically noisy variants of $40$ deterministic unconstrained problems ranging in dimension from $n=98$ to $n=125$; see Table~\ref{tab:probs}. All objective functions are sums of squares, that is, $f(\x)=\sum_{i=1}^m f_i(\x)^2$,
and are corrupted with either {\it additive} or {\it multiplicative} stochastic noise. In the former case, the noisy $f_{\ubar{\theta}}$ is given by $f_{\ubar{\theta}}(\x)=f(\x)+{\ubar{\theta}}$;  in the latter, $f_{\ubar{\theta}}(\x)=f(\x)(1+{\ubar{\theta}})$. In both cases, the centered random variable ${\ubar{\theta}}$ with standard deviation~$\sigma=10^{-3}$ is either normally or uniformly distributed. 
In order to take into account the variability due to the stochastic noise or random subspaces, $20$ replications (corresponding to random seeds common across the tests) were performed for each of the~$40$ problems; combined with the two distributions of ${\ubar{\theta}}$, we thus have a total of 1,600 {\it problem instances}. 

For all tested variants, we 
use $p$-dimensional linear interpolation models $\mhatk$, where $\nabla \mhatk = \gkhat$ is obtained from the forward finite-difference approximation from Corollary~\ref{CorQkFully}. We employ the forward finite-difference parameter $h=\min\accolade{\hopt,\dk}$, with $\hopt$ obtained following \cite{more2011edn} and provided in Table~\ref{tab:probs} and $\dk$ denoting the current trust-region radius. 
We employ such linear interpolation models so as to focus on differences due to the size and form of $\QkTranspose$ and since the corresponding trust-region subproblems can be solved exactly:
$\sk^{\star}=-\dk\frac{\gkhat}{\normii{\gkhat}}$ if $\gkhat\neq \bm{0}$ and $\sk^{\star}=\bm{0}$ otherwise.
When computing estimates of unknown function values at each iteration by means of a Monte Carlo approach using $n_k=25$ noisy function evaluations for all~$k$, available samples from previous iterations are reused, following the strategy described in the last paragraph of~\cite[Section~2.3]{audet2019stomads}, which was also used in~\cite[Section~5.2]{dzahini2020constrained}. For $\QkTranspose$, motivated by Theorems~\ref{GaussianJLTmatrix} and~\ref{sHashingTheorem}, we tested both Gaussian and $(r=1)$-Hashing strategies.
We found these to perform comparably for the tested settings and hence report
the results for the Gaussian case (labeled {\sf G-STARS-}$p$) whereby the entries of $\QkRandomTranspose\in \rpn$ are distributed as $\mathcal{N}(0,1/p)$.
We also test a STORM-like instance of STARS (labeled {\sf I-STARS-}$n$) where $\Qk=\bm{I_n}\in\rnn$. All algorithmic variants used the parameters $\gamma=2$, $\eta_2= 90\eta_1=0.9$, $\dmaxx=5$, $c_1=1$ and~$\delta_0=1$.  

All STARS variants are assessed by using data 
profiles~\cite{MoWi2009}. For each of the 1,600 noisy problem instances, let~$\bm{x_N}$ be the point with the best~$f$ function value obtained by an algorithm after $N$ evaluations of~$f_{\ubar{\theta}}$, denote by~$f^{\star}$ the least such value found by all the algorithms, and let $\bm{x_0}$ be the starting point. A problem is considered successfully solved within a convergence tolerance~$\tau\in [0,1]$ after $N$ evaluations if
\begin{equation*}
	f(\bm{x_N})\leq f^{\star} + \tau(f(\bm{x_0}) - f^{\star}).
\end{equation*}
The horizontal and vertical axes of the data profiles show, respectively, the number of noisy function evaluations divided by~$(n+1)$ and the proportion of problems solved. During the experiments a budget of 1,500$(n+1)$ noisy function evaluations is allocated to all the algorithms. For multiplicative noise, Figure~\ref{fig:dataprofsig3M} (left) shows that when the overall noise (i.e., $f(\x){\ubar{\theta}}$) is large (which is the case for many of these problems: for nearly half of these problem instances $f(\xb_0)\ge 10^3$), the {\sf G-STARS} variants with small values of $p$ perform well relative to the full space {\sf I-STARS-}$n$ variant. On the other hand, for this sampling budget and additive noise level, Figure~\ref{fig:dataprofsig3M} (right) shows that {\sf I-STARS-}$n$ is competitive with {\sf G-STARS} with $p=2, 5$. Unsurprisingly, in both cases, {\sf G-STARS-}$n$ (which uses $p=\min\accolade{n,100}$) is outperformed by the {\sf I-STARS-}$n$, which is using a better conditioned, deterministic matrix throughout.
\begin{figure}
	\centering
	\includegraphics[width=0.49\linewidth]{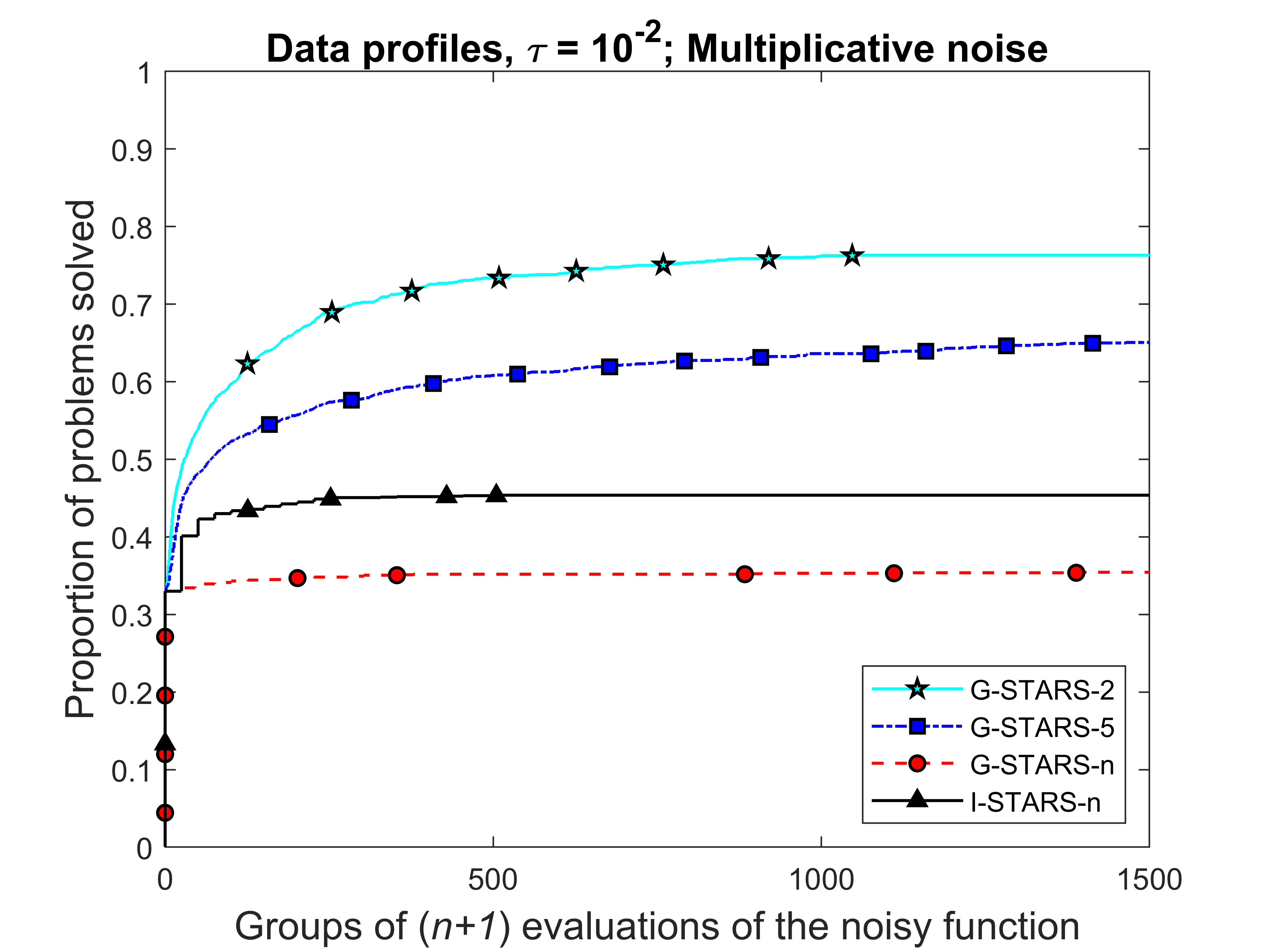} \hfill 
	\includegraphics[width=0.49\linewidth]{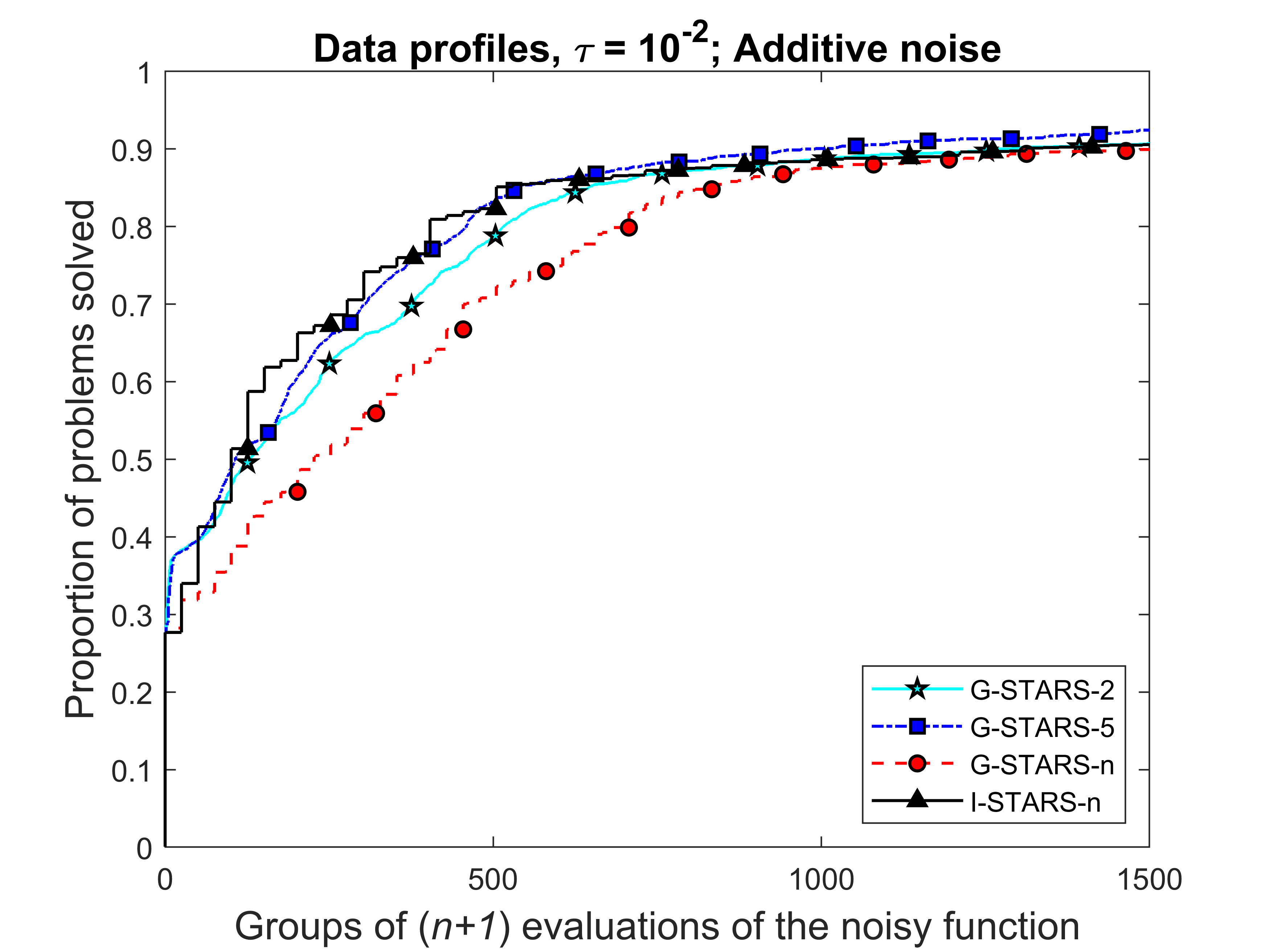}
	\caption[]{\small{Data profiles for convergence tolerance $\tau=10^{-2}$ on 1,600  problem instances for multiplicative noise (left) and additive noise (right) with standard deviation $\sigma = 10^{-3}$.}}
	\label{fig:dataprofsig3M}
\end{figure}